\documentclass[11pt,a4paper]{article}
\RequirePackage{amsmath,amssymb, amsthm}
\RequirePackage[dvipsnames,usenames]{color}
\usepackage{mathtools}

\usepackage{soul}
\usepackage{cite}

\usepackage[british]{babel}
\usepackage[latin1]{inputenc}
\usepackage[T1]{fontenc}
\usepackage[final]{showkeys} 
\usepackage{amsthm}
\usepackage{graphicx}
\usepackage{subfigure}
\usepackage{braket}
\usepackage{mathrsfs}
\usepackage{bm}
\usepackage{algorithmic}
\usepackage{hyperref}
\hypersetup{
	colorlinks=true, 
	linktoc=all,     
	linkcolor=blue,  
	citecolor=blue,
}

\newcommand{\be}{\begin{equation}} 
\newcommand{\ee}{\end{equation}}

\newcommand{\Rset}{\mathbb{R}}
\newcommand{\Cset}{\mathbb{C}}

\newcommand{\Nset}{\mathbb{N}}
\newcommand{\eu}{\ensuremath{\mathrm{e}}}
\newcommand{\iu}{\ensuremath{\mathrm{i}}}
\newcommand{\du}{\ensuremath{\mathrm{d}}}
\def\texttiny#1{{\text{\tiny{#1}}}}
\def\DC{{}^{\texttiny{C}}\! D}

\def\DS{{}^{\texttiny{S}}\! D}
\def\IS{{}^{\texttiny{S}}\! I}
\def\IRL{{}^{\texttiny{RL}}\! I}

\DeclareMathOperator{\Imag}{Im}

\theoremstyle{definition}

\newtheorem{theorem}{Theorem}

\newtheorem{proposition}[theorem]{Proposition}
\newtheorem{lemma}[theorem]{Lemma}

\theoremstyle{remark}
\newtheorem*{remark}{Remark}

\title{\bf A computational approach to exponential-type variable-order fractional differential equations}
\author{Roberto~Garrappa$^{a}$\thanks{E-mail: roberto.garrappa@uniba.it} 
$\ $ and 
Andrea~Giusti$^{b}$\thanks{E-mail: agiusti@phys.ethz.ch}
\\
\\
$^a${\em Department of Mathematics, University of Bari}
\\
{\em Via E. Orabona 4, 70125 Bari, Italy}
\\
\\
$^b$ {\em Institute for Theoretical Physics, ETH Zurich}
\\
{\em Wolfgang-Pauli-Strasse 27, 8093 Zurich, Switzerland}
}
\begin{document}
\maketitle

\begin{abstract}
We investigate the properties of some recently developed variable-order differential operators involving order transition functions of exponential type. Since the characterization of such operators is performed in the Laplace domain, it is necessary to resort to accurate numerical methods to derive the corresponding behaviors in the time domain. In this regard, we develop a computational procedure to solve variable-order fractional differential equations of this novel class. Furthermore, we provide some numerical experiments to show the effectiveness of the proposed technique.
\end{abstract}

\section{Introduction}

\newcommand\blfootnote[1]{%
	\begingroup
	\renewcommand\thefootnote{}\footnote{#1}%
	\addtocounter{footnote}{-1}%
	\endgroup
}

\blfootnote{This is the accepted version of the paper R.Garrappa, A.Giusti, \emph{A computational approach to exponential-type variable-order fractional differential equations}, J. Sci. Comp. 96 {\bf 2023} published with doi: \url{https://doi.org/10.1007/s10915-023-02283-6}. Matlab codes related to this paper are available at \url{https://www.mathworks.com/matlabcentral/profile/authors/2361481}}

Constant-order (CO) fractional operators are powerful tools used to describe various phenomena displaying non-local properties or persistent memory ({\em i.e., non-locality in time}). However, in certain scenarios the nature of non-localities can itself vary with respect to time and/or space and one is therefore forced to resort to variable-order (VO) operators (see, {\em e.g.}, \cite{DarveDeliaGarrappaGiustiRUbio2022,Giusti:2020rul,Giusti:2020kcv}). 

Over the years, several proposals for generalizing CO operators to VO have appeared in the literature. Often these attempts just consist in a simple replacement of the constant fractional order with a function of the independent variable.  It is worth noting, however, that even this simple replacement can be performed
in more than one way (see, {\em e.g.},  \cite{Zheng2022}) and, as pointed out in \cite{Samko1995a,SamkoRoss1993}, these procedures may come with some shortcomings. For VO derivatives obtained through a naive replacement, as mentioned above, it is indeed often difficult to determine the corresponding VO fractional integrals so that the operators satisfies a generalised fundamental theorem of calculus, thus providing a means to simplify theoretical and numerical analysis.

Recently, a notion of VO fractional derivative, based on the pioneering work of 
{\em Giambattista Scarpi} \cite{Scarpi1972a,Scarpi1972b} dating back to 
the early seventies, has been revived. In particular, 
in \cite{GarrappaGiustiMainardi2021} this definition of VO derivative has been reviewed and framed in terms of Luchko's theory of generalised fractional derivatives and integrals \cite{Luchko2020_FCAA,Luchko2021_Mathematics,Luchko2021_FCAA,LuchkoYamamoto2020}.

Scarpi's formulation of VO integrals and derivatives is obtained as an extension  of CO operators in the Laplace domain, as opposed to time domain. Defining and studying operators in the Laplace domain seem to simplify the theoretical treatment of these mathematical objects. Specifically, it turns out that Scarpi's fractional integrals and derivatives are, in the time domain, expressed in terms of Volterra-type integro-differential operators with weakly-singular kernels and that these kernels satisfy the Sonine condition \cite{Sonine1884}, {\em i.e.}, one of the defining requirements for the validity of a generalised fundamental theorem of calculus  \cite{DiethlemGarrappaGiustiStynes2020,Hanyga2020,SamkoCardoso2003b,SamkoCardoso2003a}.

The numerical treatment of integral and differential equations with these VO operators turns out to be more challenging. This is due to the lack of an explicit time-domain representation for Scarpi's operators. In fact, while for VO operators defined in the time domain there exists well-established numerical techniques to solve the corresponding VO fractional differential equations (VO-FDEs) (see, for instance, \cite{DuSunWang2022,SunChenLiChen2022,ZayernouriKarniadakis2015,ZengZhangKarniadakis2015,ZhaoSunKarniadakis2015,ZhengWang2020,ZhuangLiuAnhTurner2009}) the same does not apply to our case. It is therefore necessary to develop some specific approaches to tackle VO-FDEs in this new framework.

The aim of this work is to provide a further  characterisation of VO operators defined in the LT domain. In particular, we will expand on some preliminary results in \cite{GarrappaGiustiMainardi2021}, originally presented only as numerical experiments, providing precise statements and rigorous proofs. Moreover, we will investigate suitable numerical tools to solve VO-FDEs involving these operators. Specifically, we will discuss convolution quadrature rules based on Lubich's fundamental works \cite{Lubich1988a,Lubich1988b,Lubich2004}. A detailed error analysis will also be presented in order to evaluate weights of the convolution quadrature rule accurately.

Although several VO functions $\alpha(t)$ can be considered, in this work we will focus solely on those of exponential type, thus expanding and building upon previous research
\cite{GarrappaGiustiMainardi2021}. The rationale behind this choice is the fact that such a behaviour provides a smooth transition between an initial order $\alpha_1$ to a final order $\alpha_2$ and the dependence of $\alpha(t)$ on the three parameters (initial and final orders, and the transition ratio) is flexible enough to find potential applications in a wide variety of situations. Nonetheless, the general analysis presented here can be  applied, albeit with some technical differences, to other VO functions.

This work is organised as follows. In Section \ref{S:VOGeneral} we provide a brief review of Scarpi's 
VO fractional calculus. Section \ref{S:Propr_TransFunction_Exp}  describes an order transition of exponential type and analyses the main properties of the corresponding VO integral and derivative, and their implications for the VO relaxation equation. In Section \ref{S:ConvolutionQuadrature} we describe a numerical approach to solve VO-FDEs and we discuss some technical details concerning the accurate computation of weights in the proposed scheme. Lastly, in Section \ref{S:NumericalExperiments} we present some numerical experiments in order to test and validate the proposed computational approach. In Section \ref{sec:conc} we offer some concluding remarks and an outlook on future research.
\section{Variable-order fractional integrals and derivatives}\label{S:VOGeneral}
Let us begin by reviewing the key ideas of Scarpi's approach to VO calculus. To this end, we will follow \cite{GarrappaGiustiMainardi2021}, to which we refer the interested reader for a more detailed discussion on the subject.\footnote{For the sake of completeness, it is worth mentioning further results on Scarpi's operators have been presented in \cite{CuestaKiraneAlsaediAhmad2021}.}

The (CO) Riemann-Liouville integral of order $\alpha >0$ is defined in terms of a convolution integral as
\begin{equation}\label{eq:RLIntegral}
	\IRL^{\alpha}_0 f(t) = \int_0^t  \psi(t-\tau) f(\tau) \du \tau , \quad \psi(t) = \frac{t^{\alpha-1}}{\Gamma(\alpha)}  \, .
\end{equation}

Analogously, the (CO) Caputo fractional derivative of order $0<\alpha<1$ reads
\begin{equation}\label{eq:DCaputo}
	\DC^{\alpha}_0 f(t) =  \int_0^t \phi(t-\tau) f'(\tau) \du \tau, \quad
	\phi(t) = \frac{t^{-\alpha}}{\Gamma(1-\alpha) } .
\end{equation}

These two operators satisfy the so-called {\em fundamental theorem of fractional calculus}, {\em i.e.},
\begin{equation}\label{eq:CO_FundTheorCalculus}
	\DC^{\alpha}_0 \bigl[ \IRL^{\alpha}_0 f(t) \bigr] = f(t)
	\, , \quad
\end{equation}
and a generalisation of the fundamental theorem of calculus, {\em i.e.},
\begin{equation}
	\IRL^{\alpha}_0 \bigl[ \DC^{\alpha}_0 f(t) \bigr] = f(t) - f(0) \, .
\end{equation}

The validity of these properties can be directly linked to the fact that the kernels $\psi(t)$ and $\phi(t)$ satisfy the 
Sonine condition \cite{Sonine1884}, {\em i.e.}, 
\[
\int_0^t \phi(t-\tau) \psi(\tau) \du \tau = 1, \quad \forall t > 0 .
\]

Furthermore, it is worth pointing out that the Laplace transforms (LT) of $\psi(t)$ and $\phi(t)$, appearing in Eq. \eqref{eq:RLIntegral} and Eq. \eqref{eq:DCaputo}, read
\begin{equation}\label{eq:LTKernelStandard}
	\Psi(s) \coloneqq {\mathcal L} \Bigl( \psi(t) \, ; \, s \Bigr) = s^{-\alpha} 
	, \quad
	\Phi(s) \coloneqq {\mathcal L} \Bigl( \phi(t) \, ; \, s \Bigr) = s^{\alpha-1} .
\end{equation}

Now, in order to generalise the definitions in Eq. \eqref{eq:RLIntegral} and Eq. \eqref{eq:DCaputo} let 
\[
\nonumber
\begin{aligned}
	\alpha \, : \, &[0,T]& &\longrightarrow  &[0,1] \\
	& \quad t & &\longmapsto & \alpha (t)
\end{aligned}
\]
be a function such that:
\begin{description}
	\item[A1.] $\alpha$ is locally integrable on $[0,T]$;
	\item[A2.] the analytical form of the LT of $\alpha$, {\em i.e.}, $A(s) = {\mathcal L} \bigl( \alpha(t) \, ; \, s \bigr)$,
	is known;
	\item[A3.] $\lim_{t \to 0^{+}} \alpha(t) = \bar{\alpha}_0$, with $0<\bar{\alpha}_0< 1$.
\end{description}

\begin{remark}
	While A1 guarantees the existence of the LT of $\alpha$, A2 is not strictly necessary but significantly simplifies both the theoretical and numerical analyses. A3 is instead essential to define integral kernels satisfying the Sonine condition (see \cite{GarrappaGiustiMainardi2021}).
\end{remark}

Scarpi's proposal, later reformulated and extended in \cite{GarrappaGiustiMainardi2021}, sparks from the simple observation that in the constant order case $\alpha (t) = \alpha \in (0, 1)$ for all $t \in [0, T]$, then $A(s) = \alpha / s$ and, therefore, the LTs of the kernels $\psi$ and $\phi$ in Eqs. \eqref{eq:LTKernelStandard} can be rewritten as $\Psi(s) = s^{-sA(s)}$ and $\Phi(s)=s^{sA(s)-1}$. Following this line of thought, considering now a function $\alpha (t)$ that satisfies the conditions A1--A3, one can define the operators
\begin{equation}\label{eq:ScarpiIntegralConvolution}
	\IS^{\alpha(t)}_0 f(t) = \int_0^{t} \psi_{\alpha}(t-\tau) f(\tau) \du \tau 
	, \quad
	\psi_{\alpha}(t) \coloneqq {\mathcal L}^{-1} \Bigl( \Psi_{\alpha}(s) \, ; \, t \Bigr) 
\end{equation}
\begin{equation}\label{eq:ScarpiDerivativeConvolution}
	\DS^{\alpha(t)}_0 f(t) = \int_0^t \phi_{\alpha}(t-\tau) f'(\tau) \du \tau
	, \quad
	\phi_{\alpha}(t) \coloneqq {\mathcal L}^{-1} \Bigl( \Phi_{\alpha}(s) \, ; \, t \Bigr) 
\end{equation}
with kernels such that
\begin{equation}\label{eq:ScarpiKernels}
	\Psi_{\alpha}(s) \equiv s^{-s A(s)}
	, \quad 
	\Phi_{\alpha}(s) \equiv s^{s A(s)-1} .
\end{equation}

In other words, Scarpi's operators are defined in terms of kernel functions specified in the Laplace (complex) domain rather than on the time domain.  


We observe that analytical expressions for $\psi_{\alpha}(t)$ and $\phi_{\alpha}(t)$ are, in general, not known in this approach. Hence, the treatment of \eqref{eq:ScarpiIntegralConvolution} or \eqref{eq:ScarpiDerivativeConvolution} turns out to be numerical for the most part. Thus, developing numerical techniques capable of addressing these sort of problems is, in fact, among the main objectives of this work.

It is now easy to see that $\psi_{\alpha}(t)$ and $\phi_{\alpha}(t)$ naturally satisfy the Sonine equation
\[
\int_0^t \phi_{\alpha}(t-\tau) \psi_{\alpha}(\tau) \du \tau = 1, \quad \forall t > 0 ,
\] 
since, in fact, in the Laplace domain it holds that
$$
\Psi_{\alpha}(s) \Phi_{\alpha}(s) = \frac{1}{s} \, ,
$$
as one can easily check from the expressions in Eq. \eqref{eq:ScarpiKernels}. This means that $\psi_{\alpha}(t)$ and $\phi_{\alpha}(t)$ form a {\em Sonine pair} and therefore one finds that
\begin{equation}\label{eq:VO_FundTheorCalculus}
	\DS^{\alpha(t)}_0 \bigl[ \IS^{\alpha(t)}_0 f(t) \bigr] = f(t) 
	\quad \mbox{and} \quad
	\IS^{\alpha(t)}_0 \bigl[ \DS^{\alpha(t)}_0 f(t) \bigr] = f(t) - f(0) \, ,
\end{equation}
and Scarpi's operators satisfy the fundamental theorem of fractional calculus.

\section{Variable-order transitions of exponential type}\label{S:Propr_TransFunction_Exp}

The general framework described in Section \ref{S:VOGeneral} applies to any 
function $\alpha(t)$ that satisfies the assumptions A1, A2 and A3. In this section we take a closer look at VO integrals and derivatives by selecting a specific class of functions for $\alpha(t)$. 
The reason for this choice consists in the fact that the investigation of qualitative properties of both operator kernels and  solutions of FDEs with a general VO function $\alpha(t)$ would represent a rather unfeasible task unless one fixes a specific function. Moreover, selecting a specific VO function $\alpha(t)$ significantly simplifies the development of computational procedures, as we shall discuss in a forthcoming section.

We focus therefore on the case of transition functions $\alpha(t)$, defied on $\mathbb{R}^+$, describing a smooth transition from the order $0<\alpha_1<1$, at $t=0$, to $0<\alpha_2<1$ as $t \to +\infty$. More precisely, we consider a family of exponential transitions parametrised by a {\em transition rate} $c>0$ and given by
\begin{equation}\label{eq:alpha_exp}
	\alpha(t) = \alpha_2 + (\alpha_1 - \alpha_2) \eu^{-ct} ,
\end{equation} 
which satisfies assumption  A1, A2, and A3 and whose LT reads

\begin{equation}\label{eq:A_Exp}
	A(s) =  \frac{\alpha_2 c + \alpha_1 s}{s(c+s)}
	, \quad
	\Re(s) > 0 \, . 
\end{equation}

The objective of this work is to investigate VO operators with 
$\alpha(t)$ as in \eqref{eq:alpha_exp} from a purely mathematical perspective. We observe, however, that $\alpha(t)$ can be viewed as a solution of a relaxation equation, which is ubiquitous in physics.

For the sake of convenience, and to emphasise the dependence on the asymptotic orders $\alpha_1$ and $\alpha_2$ in the exponential transitions, we adopt a slightly different  notation compared to the more general one presented in \eqref{eq:ScarpiKernels}. Specifically, we denote the LTs of the generalised VO kernels by $\Psi_{\alpha_1,\alpha_2}(s)$ and $\Phi_{\alpha_1,\alpha_2}(s)$, {\em i.e.},
\[
\Psi_{\alpha_1,\alpha_2}(s) = s^{-sA(s)} = s^{-\frac{\alpha_2 c + \alpha_1 s}{c+s}}
\quad \text{and} \quad
\Phi_{\alpha_1,\alpha_2}(s) = s^{sA(s)-1} = s^{-\frac{(1-\alpha_2) c + (1-\alpha_1) s}{c+s}} .
\]

Similarly, we denote the corresponding kernels (of $\IS^{\alpha(t)}_0$ and $\DS^{\alpha(t)}_0$, respectively) in the time domain by
\[
\psi_{\alpha_1,\alpha_2}(t) = {\mathcal L}^{-1} \Bigl(\Psi_{\alpha_1,\alpha_2}(s) \, ; \, s \Bigr)
\quad \text{and} \quad
\phi_{\alpha_1,\alpha_2}(t) =  {\mathcal L}^{-1} \Bigl(\Phi_{\alpha_1,\alpha_2}(s) \, ; \, s \Bigr) \, .
\]

For the sake of a lighter notation we omit to mention explicitly the dependence of the kernels on the transition rate $c$.

Furthermore, we will refer to the usual kernels of the fractional integral and derivative of constant order $\alpha_1$ by
\[
\psi_{\alpha_1}(t) = \frac{t^{\alpha_1-1}}{\Gamma(\alpha_1)}
\quad \text{and} \quad
\phi_{\alpha_1}(t) = \frac{t^{-\alpha_1}}{\Gamma(1-\alpha_1)}	
\]
with 
\[
\Psi_{\alpha_1}(s) = {\mathcal L} \Bigl(\psi_{\alpha_1}(t) \, ; \, s \Bigr) = s^{-\alpha_1}
\quad \text{and} \quad
\Phi_{\alpha_1}(s) = {\mathcal L} \Bigl(\phi_{\alpha_1}(t) \, ; \, s \Bigr) = s^{\alpha_1-1}
\]
denoting the corresponding LTs. Similar expressions hold for the constant order $\alpha_2$.

The convergence region of the LT $A(s)$, as well as that of $\Psi_{\alpha_1}(s)$, $\Psi_{\alpha_2}(s)$, $\Phi_{\alpha_1}(s)$, and  $\Phi_{\alpha_2}(s)$, 
includes the whole positive complex half-plane. Moreover, $A(s)$ has two poles at $s=0$ and $s=-c$. Additionally, in order to make $\Psi_{\alpha_1,\alpha_2}(s)$ and $\Phi_{\alpha_1,\alpha_2}(s)$ single-valued one has to choose a branch-cut, that we set on the negative real axis.

Analytical expressions for $\psi_{\alpha_1,\alpha_2}(t)$ and $\phi_{\alpha_1,\alpha_2}(t)$ are not known. Nonetheless, a lot can be learned about these functions from their LTs. For this reason we shall first briefly review some basic properties of the LT. 

\subsection{Basic properties of the Laplace transform}

Let $f(t)$ be a piece-wise continuous function of exponential order $a$, {\em i.e.}, there exists an $M>0$ such that for some $t_0\ge0$ it is $|f(t)|\le M \eu^{at}$ for all $t\ge t_0$. Then, for any $s\in \Cset$, with $\Re(s)>a$ the LT 
\[
F(s) = {\mathcal L}\Bigl(f(t)\, , \, s \Bigr) = \int_0^{\infty} \eu^{-s t} f(t) \du t
\]
exists ({\em i.e.}, the above integral converges). Moreover, one also finds that
\[
\lim_{\Re(s) \to \infty} F(s) = 0 \, .
\]

The following standard theorems will play an important role in the derivation of the main results presented in the following. 

\begin{theorem}[Initial value (IV) theorem for the LT (see \cite{LePage1980})]\label{thm:LT_IVT}
	Let $f$ be continuous for $t>0$, of exponential order and with $f'(t)$  possessing at worst an integrable singularity at $t = 0$. Then 
	\[
	\lim_{t\to 0^+} f(t) = \lim_{s \to \infty} s F(s) .
	\]
	if the two limits exist.
\end{theorem}

\begin{theorem}[Final value (FV) theorem for the LT (see \cite{Diethelm2010,Grove1991})]\label{thm:LT_FVT}
	Let $f(t): [0,T] \to \Rset$ and assume that its LT $F(s)$ does not have any singularities in the right half-plane $\mathcal{H}_{\rm c}:=\bigl\{s \in \Cset \, | \, \Re(s) \ge 0\bigr\}$, except for possibly a simple pole at the origin. Then,
	\[
	\lim_{s \to 0^{+}} s F(s) = \lim_{t\to \infty} f(t) 
	\]
	if the two limits exist.
\end{theorem}

For a more detailed discussion of these results and of the Laplace transform method, in general, we refer the interested reader to \cite{Rasof1962}.

\subsection{Connecting VO and constant-order kernels}

In \cite{GarrappaGiustiMainardi2021}, based on some numerical experiments, it was observed that $\psi_{\alpha_1,\alpha_2}(t)$ and $\phi_{\alpha_1,\alpha_2}(t)$ tend to approach, asymptotically in $t$, the behaviour of constant-order kernels. This intuition is further supported by the following analytical results.
\begin{proposition}\label{prop:KernelIntegrRelat}
	Let $0<\alpha_1,\alpha_2 <1$ and $c>0$. Then
	\[
	\lim_{t\to0^+} \psi_{\alpha_1,\alpha_2}(t) = \lim_{t\to0^+} \psi_{\alpha_1}(t) \quad \mbox{and} 
	\quad
	\lim_{t\to\infty}
	\psi_{\alpha_1,\alpha_2}(t)
	=
	\lim_{t\to\infty}
	\psi_{\alpha_2}(t)
	\, .
	\]
\end{proposition}

\begin{proof}
	We begin by observing that
	\[
	s A(s) = \frac{\alpha_2 c + \alpha_1 s}{c+s} = \alpha_1 + B(s)
	, \quad B(s) = (\alpha_2 - \alpha_1)  \frac{c}{c+s} \, .
	\]
	Since $\Psi_{\alpha_1,\alpha_2}(s) = s^{-\alpha_1 - B(s)}$ and $\Psi_{\alpha_1}(s) = s^{-\alpha_1}$ do not have singularities in the complex half-plane $\mathcal{H}_{\rm +}:=\bigl\{s \in \Cset \, | \, \Re(s) > 0\bigr\}$, we can observe that
	$$
	\lim_{s \to \infty} \Psi_{\alpha_1,\alpha_2}(s) = \lim_{s \to \infty} s^{-B(s)} \, \Psi_{\alpha_1}(s) = \lim_{s \to \infty} \Psi_{\alpha_1}(s) \, .
	$$
	Thus,
	$$
	\lim_{s \to \infty} s \, \Psi_{\alpha_1,\alpha_2}(s) = \lim_{s \to \infty} s ( s^{-B(s)} \, \Psi_{\alpha_1}(s)) = \lim_{s \to \infty} s \, \Psi_{\alpha_1}(s) \, ,
	$$
	which implies, in light of Theorem \ref{thm:LT_IVT}, that
	$$
	\lim_{t\to0^+} \psi_{\alpha_1,\alpha_2}(t) = \lim_{t\to0^+} \psi_{\alpha_1}(t) \, .
	$$
	
	Similarly, observing that
	\[
	s A(s) 
	= \frac{\alpha_2 c + \alpha_1 s}{c+s} 
	= \alpha_2 + C(s) 
	, \quad 
	C(s) = (\alpha_1 - \alpha_2) \frac{s}{c+s} \, ,
	\]
	we can conclude that
	$$
	\lim_{s \to 0} s \, \Psi_{\alpha_1,\alpha_2}(s) = \lim_{s \to 0} s ( s^{-C(s)} \, \Psi_{\alpha_1}(s)) = \lim_{s \to 0} s \, \Psi_{\alpha_1}(s) \, ,
	$$
	and therefore, in light of Theorem \ref{thm:LT_FVT}, one has
	$$
	\lim_{t\to\infty} \psi_{\alpha_1,\alpha_2}(t) = \lim_{t\to\infty} \psi_{\alpha_1}(t) \, ,
	$$
	which concludes the proof. $\qed$
\end{proof}

A similar result can also be proven for the kernel $\phi_{\alpha_1,\alpha_2}(t)$ of the VO fractional derivative.

\begin{proposition}
	Let $0<\alpha_1,\alpha_2 <1$ and $c>0$. Then
	\[
	\lim_{t\to0^+} \phi_{\alpha_1,\alpha_2}(t) = \lim_{t\to0^+} \phi_{\alpha_1}(t) \quad \mbox{and} 
	\quad
	\lim_{t\to\infty}
	\phi_{\alpha_1,\alpha_2}(t)
	=
	\lim_{t\to\infty}
	\phi_{\alpha_2}(t)
	\, .
	\]
\end{proposition}

\begin{proof}
	Analogous to the proof of Proposition \ref{prop:KernelIntegrRelat}. $\qed$
\end{proof}

The main message of the propositions above is that VO operators realise a ``transition'' between constant-order operators.
To further highlight this point in Figure \ref{fig:Fig_KernelComp} we show, by means of numerical experiments, that the ratios
$\psi_{\alpha_1,\alpha_2}(t)/\psi_{\alpha_1} (t)$ and $\phi_{\alpha_1,\alpha_2}(t)/\phi_{\alpha_1} (t)$ settle to one at short times. The same occurs for $\psi_{\alpha_1,\alpha_2}(t)/\psi_{\alpha_2} (t)$ and $\phi_{\alpha_1,\alpha_2}(t)/\phi_{\alpha_2} (t)$ at late times. We remark that the kernels $\psi_{\alpha_1,\alpha_2}(t)$ and $\phi_{\alpha_1,\alpha_2}(t)$ have been numerically evaluated by inverting their LTs by means of the algorithm described in \cite{WeidemanTrefethen2007}.

\begin{figure}[ht]
	\begin{tabular}{cc}
		\includegraphics[width=0.44\textwidth]{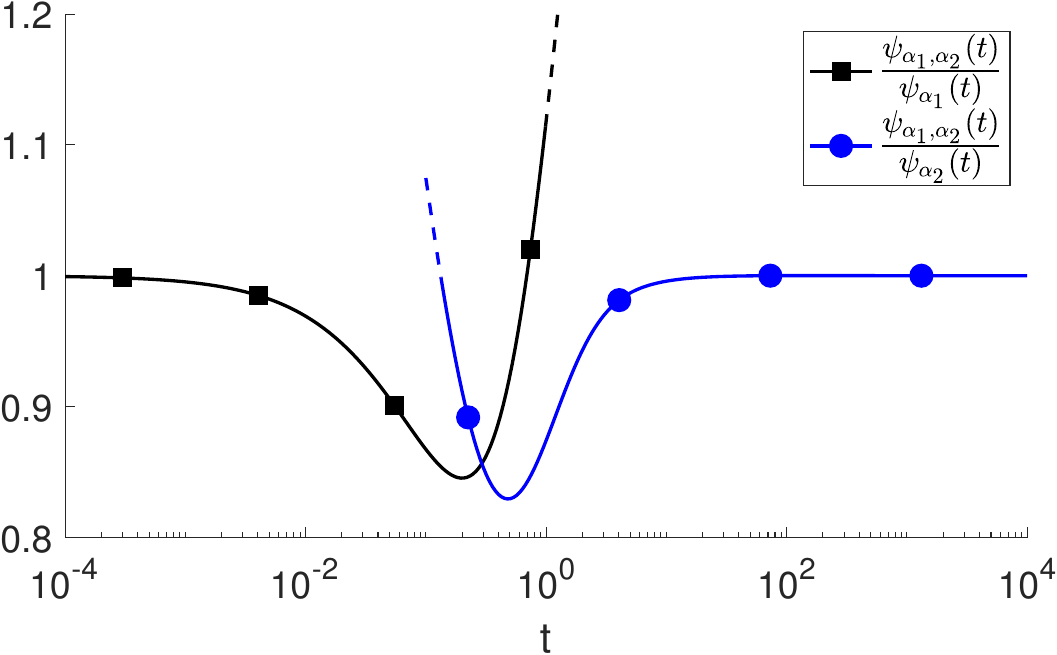} &
		\includegraphics[width=0.44\textwidth]{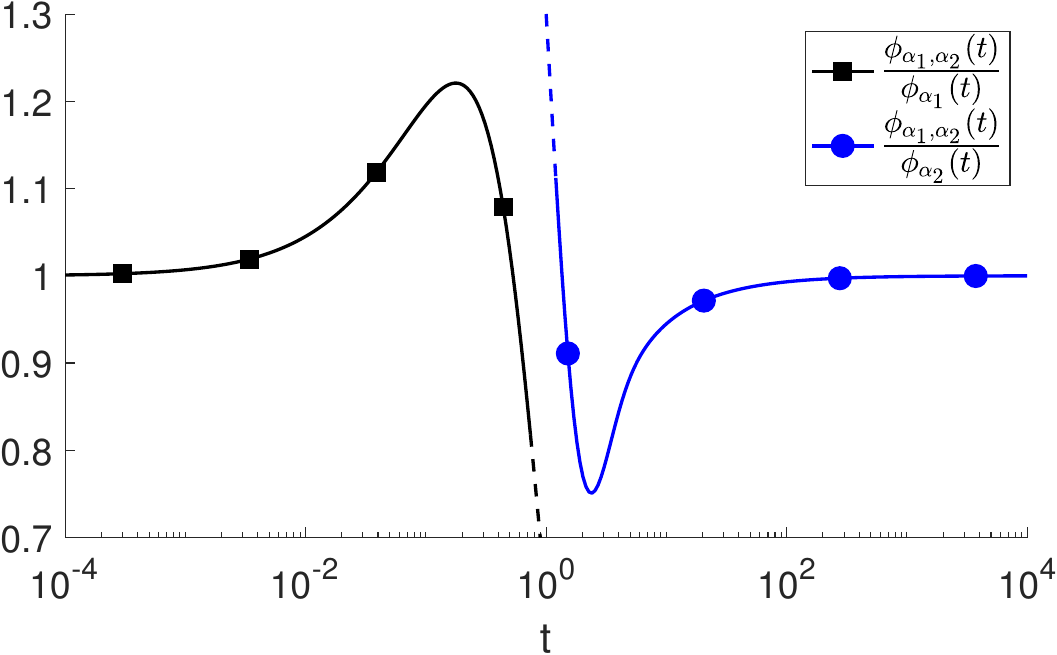} \\
	\end{tabular}
	\caption{Ratios between kernels of VO and CO integrals (right plot) and  derivatives (left plot) for small and large times; here $\alpha_1=0.6$, $\alpha_2=0.8$ and $c=2.0$.}\label{fig:Fig_KernelComp}
\end{figure}

\subsection{Relaxation equation}\label{SS:RelaxEquation}

Let us now consider the fractional relaxation equation \begin{equation}\label{eq:S_Relaxation}
	\left\{ \begin{array}{l}
		\DS^{\alpha(t)}_0 y(t) = - \lambda y(t) \\
		y(0) = y_0 \\
	\end{array} \right. ,
\end{equation}
with $ \DS^{\alpha(t)}_0$ the VO fractional derivative realizing the exponential order transition \eqref{eq:alpha_exp}, $\lambda > 0$ a real parameter and $y_0$ any positive initial value. 

In the Laplace domain Eq. \eqref{eq:S_Relaxation} can be formulated as  
$$s^{sA(s)-1} \bigl[ sY(s) - y_0 \bigr] = - \lambda Y(s) \, ,$$ 
with $Y(s)$ the LT of the solution $y(t)$. Isolating $Y(s)$ yields
\begin{equation}\label{eq:LT_Relaxation}
	Y(s) = H(s) \, y_0 \, , \quad \mbox{with} 
	\quad H(s) \coloneqq \frac{1}{s} \left( 1 + \lambda s^{-sA(s)}\right)^{-1} .
\end{equation}

The exact solutions of the corresponding relaxation equations with standard Caputo fractional derivative \eqref{eq:DCaputo} of constant order $0<\alpha_1<1$ and $0<\alpha_2<1$ are known to be respectively
\[
y_{\alpha_1}(t) = E_{\alpha_1}(-t^{\alpha_1}\lambda)y_0
\quad \text{and} \quad y_{\alpha_2}(t) = E_{\alpha_2}(-t^{\alpha_2}\lambda)y_0, 
\]
where $E_{\beta}(z)$ denotes the Mittag-Leffler function, {\em i.e.},
\[
E_{\beta}(z) = \sum_{k=0}^{\infty} \frac{z^{k}}{\Gamma(\beta k + 1)} \, .
\]

Finding a general result concerning exact solution of the VO-FDE \eqref{eq:S_Relaxation} appears to be a rather hard task. Nonetheless, we can still point out a relationship between a solution of \eqref{eq:S_Relaxation} and
the one of the corresponding CO counterparts at early and late times.

\begin{proposition}\label{prop:RelaxEq_Limit}
	Let $0<\alpha_1<1$, $0<\alpha_2<1$, $c>0$ and $\lambda >0$. There exist a function $f(t)$ such that the solution $y(t)$ of the relaxation equation \eqref{eq:S_Relaxation} satisfies 
	\[
	y(t) = E_{\alpha_1}(-t^{\alpha_1} \lambda) y_0 + f(t) y_0, \quad \lim_{t \to 0^+} f(t) = 0.
	\]
	
	Moreover, if $H(s)$  does not have singularities on $\mathcal{H}_{\rm c}$, there exists a function $g(t)$ such that
	\[
	y(t) = E_{\alpha_2}(-t^{\alpha_2} \lambda) y_0 + g(t) y_0, \quad \lim_{t \to \infty} g(t) = 0 .
	\]
\end{proposition}

\begin{proof}
	{Recall that
		$$
		H(s) = \frac{1}{s} \left( 1 + \lambda s^{-sA(s)}\right)^{-1} = \frac{s^{sA(s) -1}}{s^{sA(s)} + \lambda}
		$$
		and
		$$
		{\mathcal E}_{\alpha_1}(s;\lambda) := \mathcal{L} \left(E_{\alpha_1}(-t^{\alpha_1}\lambda) \, ; \, s\right) = \frac{s^{\alpha_1-1}}{s^{\alpha_1} + \lambda} \, .
		$$
		As in the proof of Proposition \ref{prop:KernelIntegrRelat} we write 
		$$
		s A(s) =\alpha_1 + B(s) \quad
		\mbox{and} \quad
		B(s) = (\alpha_2 - \alpha_1)  \frac{c}{c+s} \, .
		$$
		Then, one finds that
		$$
		F(s) := H(s) - {\mathcal E}_{\alpha_1}(s;\lambda) = \frac{\lambda \, s^{\alpha_1 - 1} (s^{B(s)} - 1)}{(s^{\alpha_1 + B(s)} + \lambda)(s^{\alpha_1} + \lambda)} \, .
		$$
		In other words, $F(s)$ is defined as the difference of two complex functions that admit inverse Laplace transforms and it vanishes as $|s| \to \infty$. Hence we can define $f(t) := \mathcal{L}^{-1} \left( F(s) \, ; \, t \right)$ and 
		$$
		\lim_{t \to 0^+} f(t) = 0 \, ,
		$$
		because of Theorem \ref{thm:LT_IVT}.
		
		One can easily prove the second statement again by recalling the proof of Proposition \ref{prop:KernelIntegrRelat}, specifically $s A(s) =\alpha_2 + C(s)$ and $C(s) = (\alpha_1 - \alpha_2) s/ (c+s)$, defining $G(s):= \mathcal{L} \left( g(t) \, ; \, s \right) = H(s) - {\mathcal E}_{\alpha_2}(s;\lambda)$ and applying Theorem \ref{thm:LT_FVT}.  \qed}
\end{proof}

This result is particularly important since it shows that solutions of the relaxation VO-FDE \eqref{eq:S_Relaxation} converge to solutions of the CO relaxation FDEs of order $\alpha_1$ and $\alpha_2$ respectively for $t\to 0^+$ and $t\to \infty$. This result had been already observed in \cite{GarrappaGiustiMainardi2021}, albiet only at the numerical level. In Figure \ref{fig:Fig_RelEq_Diff} we show the difference between solutions of VO and CO relaxation equations where, for simplicity we considered $y_0=1$.

\begin{figure}[ht]
	\centering
	\begin{tabular}{c}
		\includegraphics[width=0.60\textwidth]{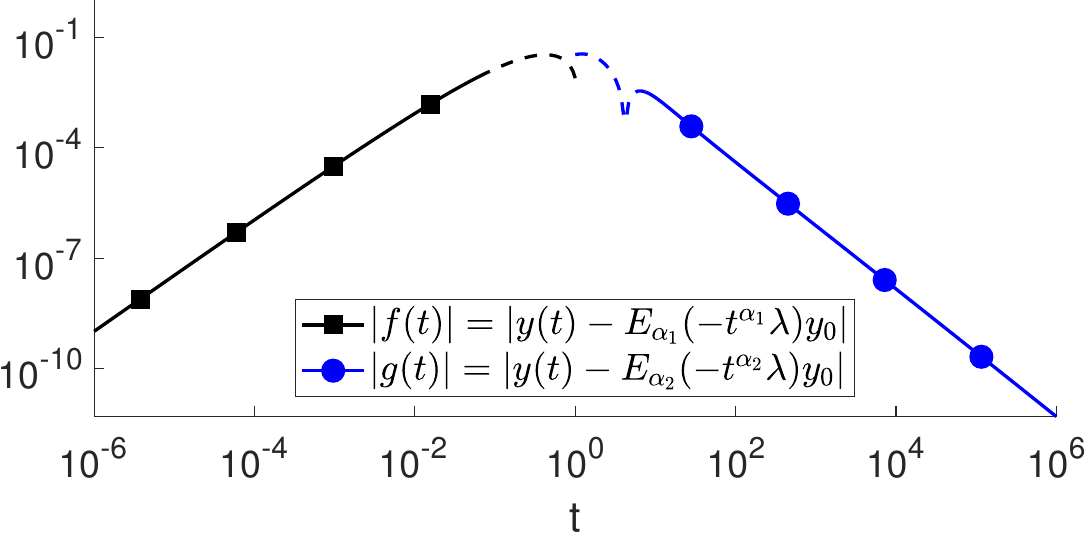} \\
	\end{tabular}
	\caption{Difference between solutions of the VO relaxation FDE \eqref{eq:S_Relaxation} and those of corresponding CO relaxation FDEs of order $\alpha_1$ and $\alpha_2$ for small and large $t$; here $\alpha_1=0.6$, $\alpha_2=0.8$, $c=2.0$, $\lambda=1.0$ and $y_0=1$.}\label{fig:Fig_RelEq_Diff}
\end{figure}

In Figure \ref{fig:Fig_Singularity_H} we show heuristically that assuming that $H(s)$ does not have singularities on $\mathcal{H}_{\rm c}$ is not particularly restrictive. Specifically, in Figure \ref{fig:Fig_Singularity_H} we plot the location of singularities  of $H(s)$,  evaluated numerically, as $\lambda$ varies in $[0.01,5]$.

\begin{figure}[ht]
	\centering
	\begin{tabular}{cc}
		\includegraphics[width=0.45\textwidth]{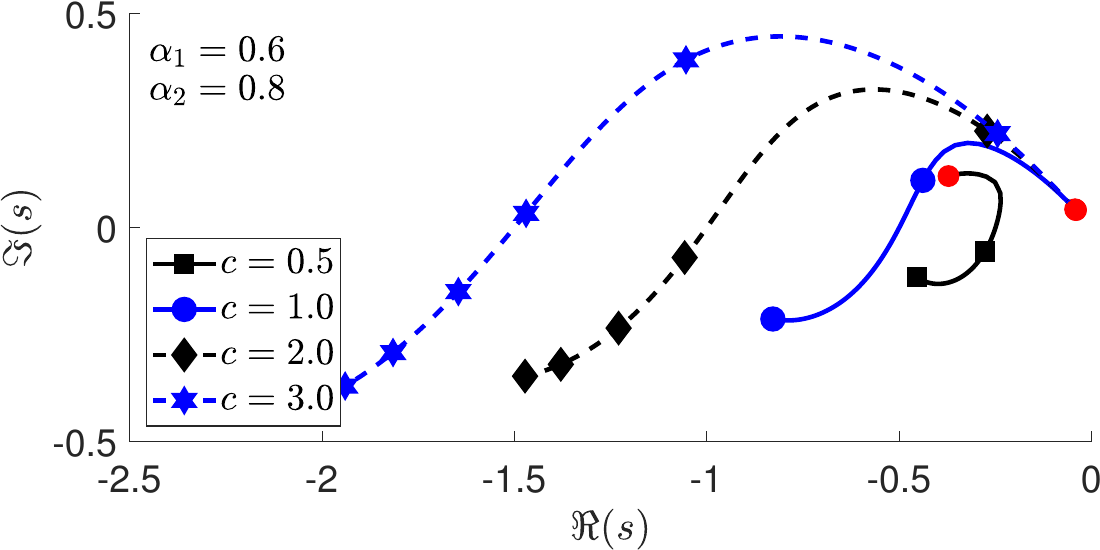} &
		\includegraphics[width=0.45\textwidth]{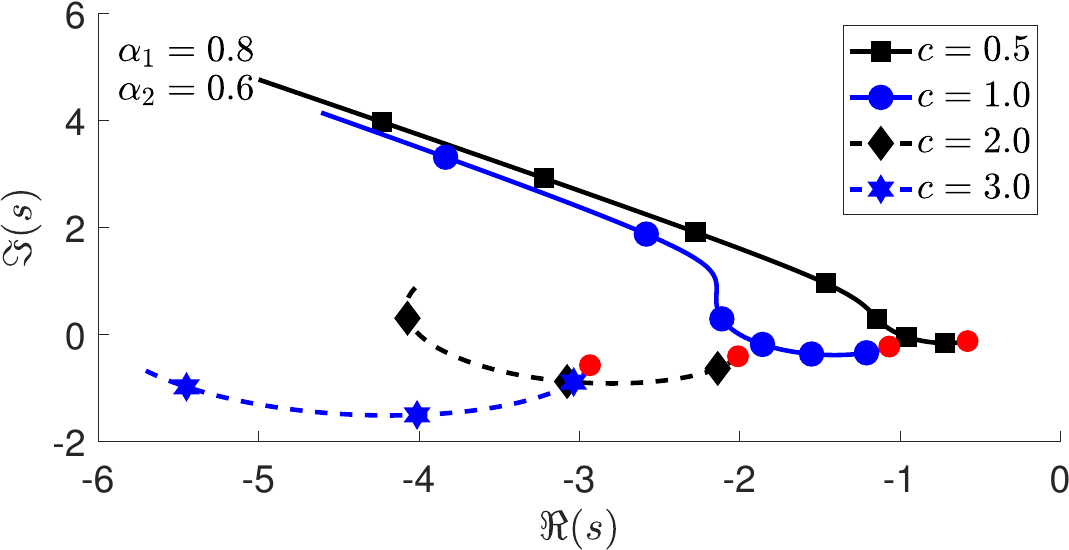} \\
		\includegraphics[width=0.45\textwidth]{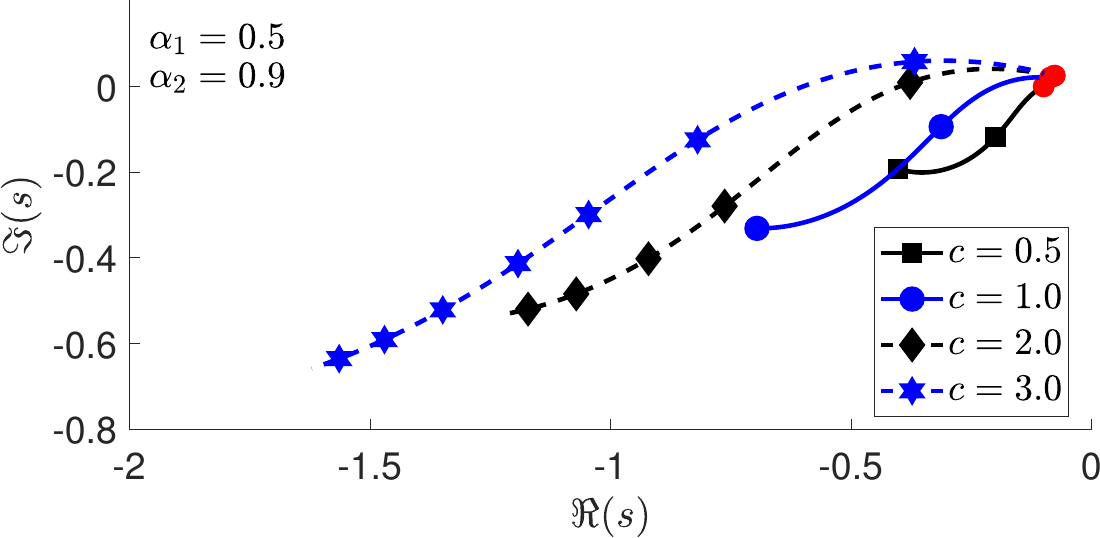} &
		\includegraphics[width=0.45\textwidth]{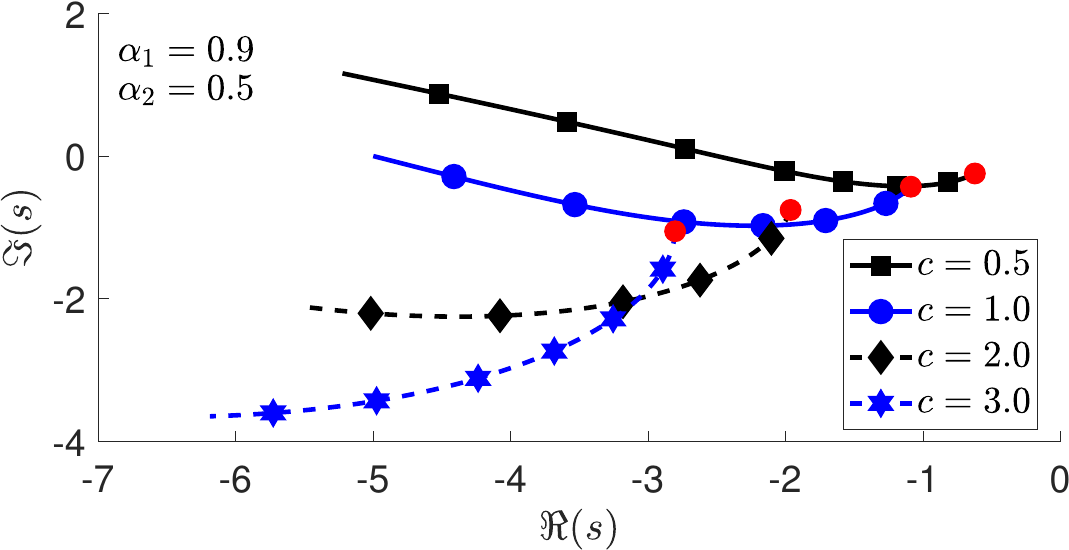} \\
	\end{tabular}
	\caption{Location of singularities of $H(s)$ for some values of $\alpha_1$, $\alpha_2$ and $c$ for $\lambda$ varying in $[0.01,5]$ (red bullets denote the starting value $\lambda=0.01$).}\label{fig:Fig_Singularity_H}
\end{figure}

\section{A numerical approach to exponential-type variable-order differential equations}\label{S:ConvolutionQuadrature}

Let us now introduce, for the exponential VO transition $\alpha(t)$ described by \eqref{eq:alpha_exp}, the VO-FDE 
\begin{equation}\label{eq:SFDE}
	\left\{\begin{array}{l}
		\DS^{\alpha(t)}_0 y(t) = f(t,y(t)), \quad t \in (0,T], \\
		y(t) = y_0, \\
	\end{array}\right. 
\end{equation}
which, after applying $\IS^{\alpha(t)}_0$ to both sides, and in view of \eqref{eq:VO_FundTheorCalculus} (and with the notation introduced in Section \ref{S:Propr_TransFunction_Exp}), can be equivalently reformulated as the  integral equation $y(t) = y_0 + \IS^{\alpha(t)}_0 f(t,y(t))$, i.e.
\begin{equation}\label{eq:SFDE_VIE}
	y(t) = y_0  + \int_0^t \psi_{\alpha_1,\alpha_2}(t-\tau) f(\tau,y(\tau))   \du \tau .
\end{equation}

The main difficulty to numerically approximate solutions of \eqref{eq:SFDE} lies in the absence of an analytical formulation of $\psi_{\alpha_1,\alpha_2}(t)$. Since just its LT $\Psi_{\alpha_1,\alpha_2}(s)$ is known, it is natural to exploit convolution quadrature rules (CQRs) introduced by Lubich in his pioneering works \cite{Lubich1988a,Lubich1988b}, and later discussed in \cite{Lubich2004}. The main feature of these rules is that they allow  to approximate the  convolution integral  in \eqref{eq:SFDE_VIE} without requiring the explicit knowledge of the convolution kernel, but rather just of its LT.

The application to \eqref{eq:SFDE_VIE} is however not straightforward. Not only because Lubich's theory requires the fulfillment of some assumptions to be valid, but also in view of some difficulties for the computation of convolution weights.

In this section, after briefly review CQRs in Lubich's framework, we will study their application in the context of VO operators and, specifically, we will devise a strategy for the accurate computation of convolution weights supported by a detailed error analysis.



\subsection{Convolution quadrature rules for variable-order operators}

CQRs introduced by Lubich provide a generalization of linear multistep methods (LMMs) for ordinary differential equations $y'(t) = f(t,y(t))$. Given a grid $t_n = nh$, with constant step-size $h>0$, a $k$-step  $(\rho,\sigma)$ LMM is defined by 
\begin{equation}\label{eq:LMM}
	\sum_{j=0}^{k} \rho_{j} y_{n-j} = h \sum_{j=0}^{k} \sigma_{j} f(t_{n-j},y_{n-j}) , 
\end{equation}
with $\rho(\xi)=\rho_{0}\xi^{k} + \rho_{1}\xi^{k-1} + \dots + \rho_{k}$ and $\sigma(\xi)=\sigma_{0}\xi^{k} + \sigma_{1}\xi^{k-1} + \dots + \sigma_{k}$ respectively the first and second characteristic polynomials and 
$\delta(\xi) = {\rho(1/\xi)}/{\sigma(1/\xi)}$
the generating function of the LMM determining its stability domain 
\[
S_{\delta} = \Cset \backslash \left\{ \delta(\xi) \, ; \, |\xi|<  1\right\} .
\]

A discrete CQR extending the  $(\rho,\sigma)$ LMM to solve \eqref{eq:SFDE_VIE}, and preserving the same convergence order, say $p$, of the LMM, reads as 
\begin{equation}\label{eq:CQR_General}
	y_n = y_0 + \sum_{j=0}^{n} \omega_{n-j} f(t_j,y_j) + \sum_{j=0}^{\nu} w_{n,j} f(t_j,y_j)
\end{equation}
where $\omega_n$ are the coefficients in the expansion of 
\begin{equation}\label{eq:CQR_Expansion}
	\Psi_{\alpha_1,\alpha_2} \Bigl( \frac{\delta(\xi)}{h} \Bigr) = \sum_{n=0}^{\infty} \omega_n \xi^n . 
\end{equation}

The second sum in \eqref{eq:CQR_General} is introduced to deal with a possible lack of smoothness of the solution at the origin (which is typical of solutions of FDEs and it is reasonably expected in this VO setting since the result in Proposition \ref{prop:KernelIntegrRelat}).

Some assumptions must be fulfilled in order to this procedure properly works and  the CQR (\ref{eq:CQR_General}) converges with the same order $p$ of the underlying LMM. In particular:
\begin{itemize}
	\item[{H1.}] $\Psi_{\alpha_1,\alpha_2}(s)$ must be analytic in a sector $\Sigma_{\gamma,d} = \bigl\{ s \in \Cset \, : \, |\arg(s-d)| < \pi - \gamma \bigl\}$, with $\gamma < \frac{\pi}{2}$ and $d\in \Rset$, and satisfying  $|\Psi_{\alpha_1,\alpha_2}(s)| \le M |s|^{-\mu}$ for some real $\mu>0$ and $M<\infty$;
	\item[{H2.}] the LMM must be $A(\theta)$-stable (namely, the stability domain $S_{\delta}$ must contain the wedge $\{z \in \Cset \,:\; |\arg(-z)|<\theta\}$ for 
	$\theta > \gamma$) and the generating function $\delta(\xi)$ must be analytic and without zeros in a neighborhood of $|\xi| \le 1$ (except for a zero at $\xi =1$).
\end{itemize}

Concerning assumption H1, we observe that  $sA(s)$ has just a pole on the negative real semi-axis (at $s=-c$), where for convenience we placed the branch-cut (see Section \ref{S:Propr_TransFunction_Exp}). Therefore $\Psi_{\alpha_1,\alpha_2}(s)$ is analytic in any sector $\Sigma_{\gamma,d}$ with $0<\gamma < \pi$ and $d>0$. Moreover, $sA(s) \to \alpha_1$ as $|s|\to \infty$ and hence $|\Psi_{\alpha_1,\alpha_2}(s)| \le M |s|^{-\mu}$ for $\mu=\alpha_1$



{Assumption H2 is satisfied by a variety of LMMs. In particular in this work we focus  on the backward Euler method $y_{n+1} = y_n + h f(t_{n+1},y_{n+1})$, whose first and second characteristic polynomials are $\rho(\xi)=\xi-1$ and $\sigma(\xi) = \xi$ and the corresponding generating function is hence $\delta(\xi)  = 1-\xi$. 
	
	Since the method is expected to inherit the convergence order $p=1$ of the backward Euler method \cite[Theorem 3.1]{Lubich1988a}, there is no need of introducing starting weights in (\ref{eq:CQR_General}) and therefore its extension to the VO integral equation \eqref{eq:SFDE_VIE} simply reads as
	\begin{equation}\label{eq:GL_Conv_General}
		y_n = y_0 + \sum_{j=0}^{n} \omega_{n-j} f(t_j,y_j) ,
	\end{equation}
	where the convolution weights $\omega_n$ are the coefficients in the expansion of
	\begin{equation}\label{eq:GL_VO_Weights_Expansion}
		\widehat\Psi_{\alpha_1,\alpha_2}^{[h]}(\xi) \coloneqq \Psi_{\alpha_1,\alpha_2} \left( \frac{1-\xi}{h} \right) 
		= \sum_{n=0}^{\infty} \omega_n \xi^n .
	\end{equation}

}

{\begin{remark}
		It can be of interest to observe that in the CO case the CQR obtained by exploiting the generating function of the backward Euler method has the same weights of the Gr\"unwald-Letnikov (GL) discretization scheme (we refer, for instance,  to \cite{Diethelm2010,GorenfloAbdelRehim2007,SamkoKilbasMarichev1993,SchererKallaTangHuang2011} for more details about this scheme). The method proposed here can be in some sense considered as a sort of generalization of the GL scheme to VO operators \eqref{eq:ScarpiIntegralConvolution} and \eqref{eq:ScarpiDerivativeConvolution}.
	\end{remark}
}

\subsection{Computation of weights}\label{S:VO_GL_EvaluationWeights}

Computing convolution weights $\omega_n$ in \eqref{eq:GL_Conv_General} is a challenging task which must be accomplished with high precision to avoid loss of accuracy and, possibly, in a fast way. 


From \eqref{eq:GL_VO_Weights_Expansion} we observe that $\omega_n$ are coefficients in the Taylor expansion of $\widehat\Psi_{\alpha_1,\alpha_1}^{[h]}(\xi)$ around the origin. Therefore, they can be expressed in terms of Cauchy integrals
\[
\omega_n = \frac{1}{n!} \frac{\du^n}{\du \xi^n} \widehat\Psi_{\alpha_1,\alpha_2}^{[h]}(\xi) \Bigl|_{\xi=0}
= \frac{1}{2 \pi \iu} \int_{{\mathcal C}} \frac{\widehat\Psi_{\alpha_1,\alpha_2}^{[h]}(z)}{z^{n+1}} \du z
, \quad n = 0, 1, \dots,
\]
over contours  ${\mathcal C}$ encompassing $\xi=0$ in the region of analiticity of  $\widehat\Psi_{\alpha_1,\alpha_2}^{[h]}(\xi)$.

The function $\widehat\Psi_{\alpha_1,\alpha_2}^{[h]}(\xi)$ is analytic the whole complex plane $\Cset$ except the real line $[1,+\infty)$. Hence, we select ${\mathcal C}$ as a circle with center at the origin and radius $0<\rho < 1$. Since ${\mathcal C} = \bigl\{ z \in \Cset \, | \, z = \rho \eu^{\iu \theta}, \, \theta \in [0, 2\pi]\}$, it is
\begin{equation}\label{eq:WeightsIntegralCauchyCircle}
	\omega_n = \frac{\rho^{-n}}{2 \pi} \int_{0}^{2\pi} \Psi_{\alpha_1,\alpha_2}^{[h]}\bigl(\rho \eu^{\iu \theta}\bigr) \eu^{-\iu n \theta} \du \theta 
	, \quad n = 0, 1, \dots
\end{equation}
and the application of the compound trapezoidal quadrature rule on $L$ nodes $\theta_k = 2 \pi k/L$, $k=0,1,\dots,L-1$, provides the approximations
\begin{equation}\label{eq:WeightsQuadratureRule}
	\omega_n^{(\rho,L)} = \frac{\rho^{-n}}{L} \sum_{k=0}^{L} \Psi_{\alpha_1,\alpha_2}^{[h]}\bigl(\rho \eu^{\iu 2 \pi k /L}\bigr) \eu^{-\iu 2 \pi k n/L } 
	, \quad n = 0, 1, \dots
\end{equation}
a formula which lends itself to being calculated by means of an efficient  fast Fourier transform algorithm.


The quadrature rule \eqref{eq:WeightsQuadratureRule} depends on two parameters: the contour radius $\rho$ and the number $L$ of quadrature nodes. They are strictly related since their ratio determines the step-size of the quadrature rule and, hence, the discretization error. However, in the finite precision arithmetic of computers, the parameter $\rho$ plays a further important role: it indeed determines the distance of  the integration contour from singularities of $\widehat\Psi_{\alpha_1,\alpha_2}^{[h]}(z)$ and choosing a contour too  close to the singularities may leads to strong round-off errors. It is therefore necessary to separately study discretization and round-off errors and take both of them into account for selecting parameter $\rho$ and $L$. The aim is minimizing the error and  taking it below a given tolerance $\tau>0$ (possibly close to the precision machine), namely to obtain
\[
|\omega_n  - \omega_n^{(\rho,L)}|<\tau, \quad n = 0, 1, \dots N.
\]


\subsection{Discretization error}
The main error source for computation of weights $\omega_n$ is the discretization error in the quadrature rule \eqref{eq:WeightsQuadratureRule}. In this respect we are able to provide the following result.

\begin{proposition}
	Let $0<\rho<r<1$ and $a=\log(r/\rho)$. Then for any $n \in \Nset$ and $L>0$ it is
	\begin{equation}\label{eq:ErrorWeights}
		|\omega_n  - \omega_n^{(\rho,L)}| \le \frac{M_{r,\rho} \rho^{-n}}{r^L\rho^{-L} - 1}
		, \quad
		M_{r,\rho} = \max_{\Imag z \ge -a } \bigl| \widehat\Psi_{\alpha_1,\alpha_2}^{[h]}\bigl(z\bigr) \bigr| .
	\end{equation}
\end{proposition}

\begin{proof}
	Consider the $2\pi$-periodic function $v_n(\theta) = \widehat\Psi_{\alpha_1,\alpha_2}^{[h]}\bigl(\rho \eu^{\iu \theta}\bigr)  \eu^{-\iu n \theta}$, such that $\omega_n = \frac{\rho^{-n}}{2\pi} \int_{0}^{2\pi} v_n(\theta) \du \theta$. Its complex extension 
	
	\[
	v_n(x + \iu y) = \widehat\Psi_{\alpha_1,\alpha_2}^{[h]}\bigl(\rho \eu^{\iu (x+\iu y)}\bigr) \eu^{-\iu n (x+\iu y)}
	=  \widehat\Psi_{\alpha_1,\alpha_2}^{[h]}\bigl(\rho \eu^{-y}  \eu^{\iu x}\bigr) \eu^{n y -\iu n x } 
	\]
	is analytic whenever $\rho \eu^{-y}<r<1$, namely in the strip $y>-a$. We can therefore apply \cite[Theorem 3.1]{TrefethenWeideman2014} to obtain $	|\omega_n  - \omega_n^{(\rho,L)}| \le {M_{r,\rho} \rho^{-n} }/({\eu^{a L} - 1})$, from which the proof follows. \qed
	
\end{proof}

\subsection{Round-off errors}

The discretization error \eqref{eq:ErrorWeights} is not the only error affecting the accuracy of computed weights $\omega_n$. When high accuracy and/or a large number of weights is requested, round-off errors due the finite-precision arithmetic of computers must be taken into account.

To introduce round-off errors in our analysis, we represent weights which are actually computed as 
\[
\widehat\omega_n^{(\rho,L)} = \frac{\rho^{-n}}{L} \sum_{k=0}^{L} \widehat\Psi_{\alpha_1,\alpha_2}^{[h]}\bigl(\rho \eu^{\iu 2 \pi k /L}\bigr) \eu^{-\iu 2 \pi k n/L } (1+ \epsilon_k) ,
\]
with $|\epsilon_k|\approx \epsilon$ and $\epsilon$ the precision machine. Thus, if we consider
\[
\widehat{M}_{\rho} = \max_{k=0,L} \bigl|\widehat\Psi_{\alpha_1,\alpha_2}^{[h]}\bigl(\rho \eu^{\iu 2 \pi k /L}\bigr) \eu^{-\iu 2 \pi k n/L }\bigr| = \max_{k=0,L} \bigl|\widehat\Psi_{\alpha_1,\alpha_2}^{[h]}\bigl(\rho \eu^{\iu 2 \pi k /L}\bigr)\bigr|, 
\]
since $0<\rho<1$, and in view of Lemma \ref{lem:Psi_BoundCircle} below, we obtain the following rough estimation 
\[
|\omega_n^{(\rho,L)}  - \widehat\omega_n^{(\rho,L)}|  \le \rho^{-n} \widehat{M}_{\rho} \epsilon \le \rho^{-N} M_{r}  \epsilon
, \quad M_{r} = \bigl|\Psi_{\alpha_1,\alpha_2}^{[h]}(r)\bigr|,
\]
holding for any $n = 0,1,\dots, N$ and for a reasonably small step-size $h<1-r$.

The overall error is hence
\[
|\omega_n - \widehat\omega_n^{(\rho,L)}| \le \underbrace{|\omega_n - \omega_n^{(\rho,L)} |}_{\text{discret. error}} + \underbrace{| \omega_n^{(\rho,L)} - \widehat\omega_n^{(\rho,L)}|}_{\text{round-off error}} ,
\]
and the goal is to keep the round-off error smaller than the discretization error $\tau$, thus to possibly neglect it. We can therefore fix a safety factory $0<F_{\text{s}}<1$, for instance $F_{\text{s}}\approx 0.01\div0.1$, and impose the round-off error to be proportional to $F_{\text{s}} \tau$, thus to select
\begin{equation}\label{eq:Radius_RoundOff_Estimation}
	\rho^{-N} M_r \epsilon = F_{\text{s}} \tau
	\quad \Rightarrow \quad
	\rho = \exp\Bigl( - \frac{1}{N}  \log \frac{F_{\text{s}} \tau}{M_{r} \epsilon}\Bigr) .
\end{equation}

\subsection{A strategy for parameters selection}

The proposed strategy to select parameters $\rho$, $r$ and $L$ is to first fix a value for $r<1$, which for convenience must be chosen close to 1, and determine by \eqref{eq:Radius_RoundOff_Estimation} a suitable contour radius $\rho<r$ on the basis of the number $N$ of nodes which must be computed. When \eqref{eq:Radius_RoundOff_Estimation} leads to a value of $\rho \ge r$, the value of $r$ must be suitably increased.

The tolerance $\tau$ should be chosen as small as possible to consider weights as computed in a virtually exact way, but not too strict to impose a huge number of quadrature nodes to achieve it. Therefore, $\tau$ is selected in the range $10^{-12}\div 10^{-13}$.

Once $\rho$ is selected, it is necessary to provide an estimation for the value of $M_{r,\rho}$ in \eqref{eq:ErrorWeights}. It seems not possible to provide a theoretical estimation for $M_{r,\rho}$ and therefore it will be necessary to sample  $\widehat\Psi_{\alpha_1,\alpha_2}^{[h]}\bigl(z\bigr)$ on a sufficiently large number of points $z \in \Cset$, with $\Imag z \ge -\log(r/\rho)$ and determine its maximum.

Finally, it is simple to use \eqref{eq:ErrorWeights} to obtain the number $L$ of nodes in the trapezoidal quadrature rule \eqref{eq:WeightsQuadratureRule} to to achieve the accuracy $|\omega_n  - \omega_n^{(\rho,L)}|<\tau$ by selecting
\[
L \ge \frac{\log(M_{r,\rho} \rho^{-n} + \tau)-\log \tau}{\log r - \log \rho} .
\]

\subsection{Two technical results}

In this subsection we collect two technical results which has been used for the estimation of round-off error in the previous subsection. We confine their presentation here with the only aim of lightening the reading of the paper.

\def\zR{z_{\texttiny{R}}}
\def\zI{z_{\texttiny{I}}}
\begin{lemma}\label{lem:Psi_ComplexRepresentation}
	Let $\xi = x +\iu  y  \in \Cset$. Then 
	\begin{equation}\label{eq:Psi_ComplexRepresentation}
		\bigl| \widehat\Psi_{\alpha_1,\alpha_2}^{[h]}(\xi) \bigr| = \exp \left(  A(x,y) \left( \ln h -\frac{1}{2} \ln \bigl((1-x)^2+y^2) \right) + \theta_{\xi}  B(x,y) \right) ,
	\end{equation}
	where $\theta_{\xi} = \arg(1-\xi)$ and
	\[
	\begin{aligned}
		A(x,y) &= \frac{\Bigl[ \alpha_1 y^2 + \bigl(ch+(1-x)\bigr)\bigl(\alpha_1(1-x)+\alpha_2 ch\bigr)\Bigr]}{\bigl(ch+(1-x)\bigr)^2 + y^2} ,\\
		B(x,y) &= \frac{(\alpha_2 - \alpha_1) ch y }{\bigl(ch+(1-x)\bigr)^2 + y^2}. \\
	\end{aligned}
	\]
\end{lemma}

\begin{proof}
	First rewrite $\Psi_{\alpha_1,\alpha_2}^{[h]}(\xi) $ as
	\[
	\Psi_{\alpha_1,\alpha_2}^{[h]}(\xi) 
	= \exp\Bigl(- \frac{\alpha_2 c h + \alpha_1 (1-\xi)}{ch + 1 - \xi} \Bigl[ \ln (1-\xi) - \ln h\Bigr] \Bigr) 
	\]
	and observe that $A(x,y)$ and $B(x,y)$ are, respectively, the real and the imaginary part of the first term of the argument in the above exponential, i.e.
	\[
	\frac{\alpha_2 c h + \alpha_1 (1-\xi)}{ch + 1 - \xi} =
	\frac{\alpha_2 c h + \alpha_1  - \alpha_1 x - \iu \alpha_1 y }{ch + 1 - x - \iu y} =
	A(x,y)  + \iu B(x,y) . 
	\]
	
	Hence, after evaluating 
	\[
	\ln (1-\xi) = \ln |1-\xi| + \iu \theta  = \frac{1}{2} \ln \bigl((1-x)^2+y^2) + \iu \theta 
	,
	\] 
	the proof follows by means of elementary manipulations. \qed
\end{proof}

\begin{lemma}\label{lem:Psi_BoundCircle}
	Let  $z = \rho \eu^{\iu \theta}  \in \Cset$, $0<\rho < r < 1$, $-\pi\le \theta < \pi$ and $0<h<1-r$. Then 
	\[
	\bigl| \Psi_{\alpha_1,\alpha_2}^{[h]}(z) \bigr| \le \bigl|\Psi_{\alpha_1,\alpha_2}^{[h]}(r)\bigr|.
	\]
\end{lemma}

\begin{proof}
	We assume $\alpha_1 < \alpha_2$; when $\alpha_2 < \alpha_1$ the proof is made in a symmetric way. 
	Denote with $x$ and $y$ respectively the real and imaginary parts of $z$, i.e. $x = \rho \cos \theta$ and $y= \rho \sin \theta$ and consider, from Lemma \ref{lem:Psi_ComplexRepresentation} the representation \eqref{eq:Psi_ComplexRepresentation} for $\bigl| \Psi_{\alpha_1,\alpha_2}^{[h]}(z) \bigr|$. It is easy to observe that 
	\[
	A(x,y) \le A(x,0) \le \frac{ \alpha_1(1-\rho)+\alpha_2 ch}{ch + 1-\rho} = A(\rho,0) < A(r,0)
	\]
	and $A(x,y) \ge 0$ since $x<1$. Moreover, $\ln \bigl((1-x)^2+y^2\bigr) = 1 - 2 \rho \cos \theta + \rho^2$ and hence
	\[
	\ln h -\frac{1}{2} \ln \bigl((1-x)^2+y^2) \le \ln h - \ln(1-\rho) < h - \ln(1-r) 
	\]	
	and $\ln h - \ln(1-r) \le 0$ since we assumed $h < 1-r$. Therefore, 
	\[
	0 \ge A(x,y) \left( \ln h -\frac{1}{2} \ln \bigl((1-x)^2+y^2) \right) \ge A(r,0) \left( \ln h - \ln(1-r) \right) .
	\]
	
	Let $\theta_z = \arg(1-z)$ and, since $\alpha_1 < \alpha_2$, it is $\theta_z B(x,y) \le 0$. Indeed, when $0\le\theta<\pi$ it is $y \ge 0$ and $\theta_z \le 0$, whilst when  $-\pi\le\theta<0$ it is $y \le 0$ and $\theta_z \ge 0$ (in addition it is always $c>0$ and $h>0$). 	Therefore, from \eqref{eq:Psi_ComplexRepresentation}  we obtain
	\[
	\bigl| \Psi_{\alpha_1,\alpha_2}^{[h]}(z) \bigr| \le \exp \left( A(r,0) \left( \ln h - \ln(1-r) \right) \right) =  \bigl|\Psi_{\alpha_1,\alpha_2}^{[h]}(r) \bigr|
	\]
	which allows to conclude the proof. \qed
\end{proof}

\section{Numerical experiments}\label{S:NumericalExperiments}

\blfootnote{Matlab codes employed in this Section are available at \url{https://www.mathworks.com/matlabcentral/fileexchange/133232-fde_vo_exponential}}

To validate the numerical scheme discussed in the previous section, we perform here some tests. {Although the  scheme can be applied, in principle, to any linear or nonlinear problem, we restrict experiments to just some linear problems since the nonlinear case could require a more detailed analysis.}  

It is quite natural to start by testing the method with the relaxation equation \eqref{eq:S_Relaxation} for which a reference solution can be approximated, with high accuracy, by numerical inversion of the LT as explained, and already exploited, in Subsection \ref{SS:RelaxEquation} (it is worth mentioning that no exact solutions are in general available with VO-FDEs).

For different sets of values $\alpha_1$, $\alpha_2$, $c$ and $\lambda$, in Figure \ref{fig:Fig_RelaxEq} we show solutions $y^{\star}(t)$ obtained by inversion of the LT and used as reference solutions.

\begin{figure}[ht]
	\centering
	\includegraphics[width=0.60\textwidth]{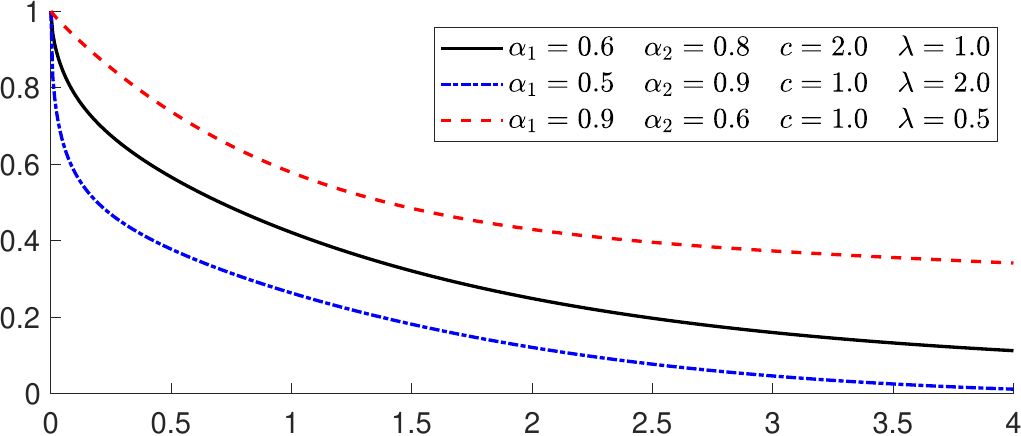}
	\caption{Solutions of the relaxation equation \eqref{eq:S_Relaxation} for some sets of parameters.}\label{fig:Fig_RelaxEq}
\end{figure}

In Table \ref{tab:Error_EOC_RelaxEq} we collect, for a decreasing sequence of the step-size $h$, errors $E(h)$ at $T=4$ between  numerical solutions $y_h(T)$ and references $y^{\star}(T)$, together with the Estimated Order of Convergence (EOC) evaluated as
\[
\text{EOC} = \log_2 \frac{ E(h) }{ E(h/2) }, \quad
E(h)=|y^{\star}(T)-y_h(T)|.
\]

\begin{table}[ht]
	\footnotesize
	\[
	\begin{array}{|r|cc|cc|cc|} \hline
		& \multicolumn{2}{|c|}{\alpha_1=0.6 \;\; \alpha_2=0.8}& \multicolumn{2}{|c|}{\alpha_1=0.5 \;\; \alpha_2=0.9}& \multicolumn{2}{|c|}{\alpha_1=0.9 \;\; \alpha_2=0.6}\\ 
		& \multicolumn{2}{|c|}{c=2.0 \;\; \lambda=1.0}& \multicolumn{2}{|c|}{c=1.0 \;\; \lambda=2.0}& \multicolumn{2}{|c|}{c=1.0 \;\; \lambda=0.5}\\ 
		h & \textrm{Error} & \textrm{EOC} & \textrm{Error} & \textrm{EOC} & \textrm{Error} & \textrm{EOC} \\ \hline
		2^{-2} & 9.96(-3) & & 1.02(-2) & & 3.71(-3) & \\ 
		2^{-3} & 4.97(-3) & 1.004 & 5.14(-3) & 0.981 & 1.67(-3) & 1.146 \\ 
		2^{-4} & 2.48(-3) & 1.003 & 2.59(-3) & 0.991 & 7.89(-4) & 1.085 \\ 
		2^{-5} & 1.24(-3) & 1.002 & 1.30(-3) & 0.995 & 3.82(-4) & 1.046 \\ 
		2^{-6} & 6.18(-4) & 1.001 & 6.50(-4) & 0.998 & 1.88(-4) & 1.024 \\ 
		2^{-7} & 3.09(-4) & 1.000 & 3.25(-4) & 0.999 & 9.31(-5) & 1.012 \\ 
		\hline
	\end{array}
	\]
	\normalsize
	\caption{Errors (at $T=4$) and EOCs of the numerical method applied to the relaxation equation \eqref{eq:S_Relaxation}} 
	\label{tab:Error_EOC_RelaxEq} 
\end{table}

As we can clearly observe, the test confirms that, as predicted from theory, the scheme preserves the first order convergence in the VO case as well. 

{We introduce now in \eqref{eq:S_Relaxation} a forcing term to obtain the VO-FDE
	\begin{equation}\label{eq:LinearForced}
		\left\{\begin{array}{l}
			\DS^{\alpha(t)}_0 x(t) = \lambda x(t) + \sin t \\
			x(t) = x_0, 
		\end{array}\right. 
		, \quad t \in (0,T]. \\
	\end{equation}
	
	The reference solutions, presented in Figure \ref{fig:Fig_LinearForced}, are now  evaluated by the same method proposed in this work but with the smaller step-size $h=2^{-10}$.
}

\begin{figure}[ht]
	{\centering
		\includegraphics[width=0.60\textwidth]{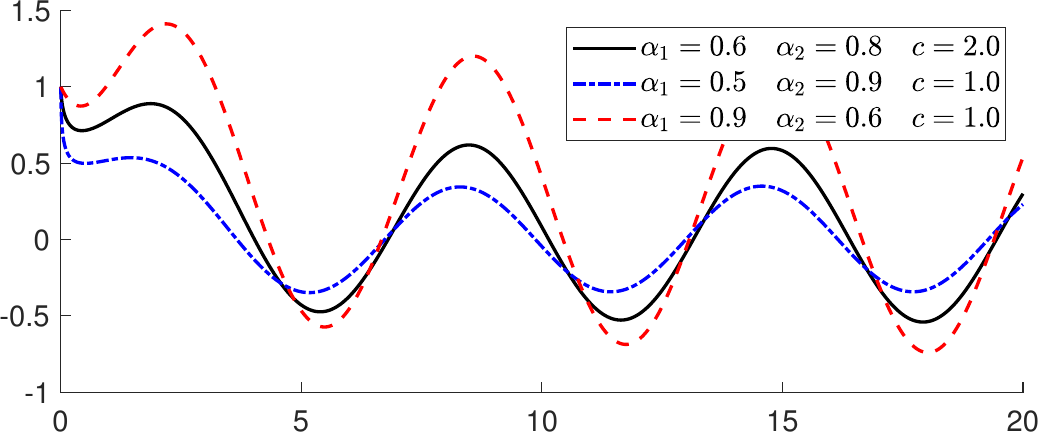}
		\caption{Solutions of the linear equation with forcing term \eqref{eq:LinearForced} for some sets of parameters.}\label{fig:Fig_LinearForced}
	} 
\end{figure}

{Also with the test problem \eqref{eq:LinearForced} the comparison between reference and numerical solutions allows to confirm the first-order convergence as one can clearly observe from Table \ref{tab:Error_EOC_LinearForced}.
}

\begin{table}[htb]
	{\footnotesize
		\[
		\begin{array}{|r|cc|cc|cc|} \hline
			& \multicolumn{2}{|c|}{\alpha_1=0.6 \;\; \alpha_2=0.8}& \multicolumn{2}{|c|}{\alpha_1=0.5 \;\; \alpha_2=0.9}& \multicolumn{2}{|c|}{\alpha_1=0.9 \;\; \alpha_2=0.6}\\ 
			& \multicolumn{2}{|c|}{c=2.0}& \multicolumn{2}{|c|}{c=1.0}& \multicolumn{2}{|c|}{c=1.0}\\ 
			h & \textrm{Error} & \textrm{EOC} & \textrm{Error} & \textrm{EOC} & \textrm{Error} & \textrm{EOC} \\ \hline
			2^{-2} & 1.88(-2) & & 9.74(-3) & & 1.32(-2) & \\ 
			2^{-3} & 9.60(-3) & 0.972 & 5.05(-3) & 0.948 & 6.21(-3) & 1.087 \\ 
			2^{-4} & 4.82(-3) & 0.994 & 2.56(-3) & 0.981 & 2.98(-3) & 1.059 \\ 
			2^{-5} & 2.39(-3) & 1.014 & 1.27(-3) & 1.007 & 1.44(-3) & 1.049 \\ 
			2^{-6} & 1.16(-3) & 1.043 & 6.19(-4) & 1.039 & 6.90(-4) & 1.061 \\ 
			2^{-7} & 5.42(-4) & 1.097 & 2.90(-4) & 1.095 & 3.21(-4) & 1.107 \\ 
			\hline
		\end{array}
		\]
		\normalsize
		\caption{Errors (at $T=20$) and EOCs of the numerical method applied to the linear equation with forcing term \eqref{eq:LinearForced}} 
		\label{tab:Error_EOC_LinearForced} 
	}
\end{table}

{We now build the $2$-dimensional linear system of  VO-FDEs
	\begin{equation}\label{eq:LinearSystem}
		\left\{\begin{array}{l}
			\DS^{\alpha(t)}_0 x(t) = p^{-} x(t) + y(t) + 1 \\
			\DS^{\alpha(t)}_0 y(t) = -2 x(t) + p^{+} y(t) + 2 , \quad t \in (0,T], \\
			x(t) = x_0, \, \, y(t) = y_0,
		\end{array}\right. 
	\end{equation}
	by selecting $p^{\pm} = \cos \theta \pm \sqrt {\cos^2 \theta +1 }$ and choosing  the parameter $\theta$ so that when the same system has CO derivatives of order $\alpha_1$ and $\alpha_2$, it shows a stable behavior. In particular, $\theta$ is selected in order to eigenvalues of the matrix system of \eqref{eq:LinearSystem} lie in  the sector $|\arg(z)| > \sigma$, $\sigma =\max\{\alpha_1,\alpha_2\}\pi/2$, and we used  the value $\theta = 1.1 \cdot \sigma$ in all experiments.

	The reference solutions are presented in Figure \ref{fig:Fig_LinearSystem}. We have presented here a comparison with the solutions of the same system with CO derivatives of order $\alpha_1$ and $\alpha_2$ to show how the behaviour for small $t$ (the box in each plot) and for large $t$ follows the behaviour of solutions of the CO system with order $\alpha_1$ and $\alpha_2$ respectively.
}

\begin{figure}[ht]
	{\centering
		\begin{tabular}{cc}
			\includegraphics[width=0.45\textwidth]{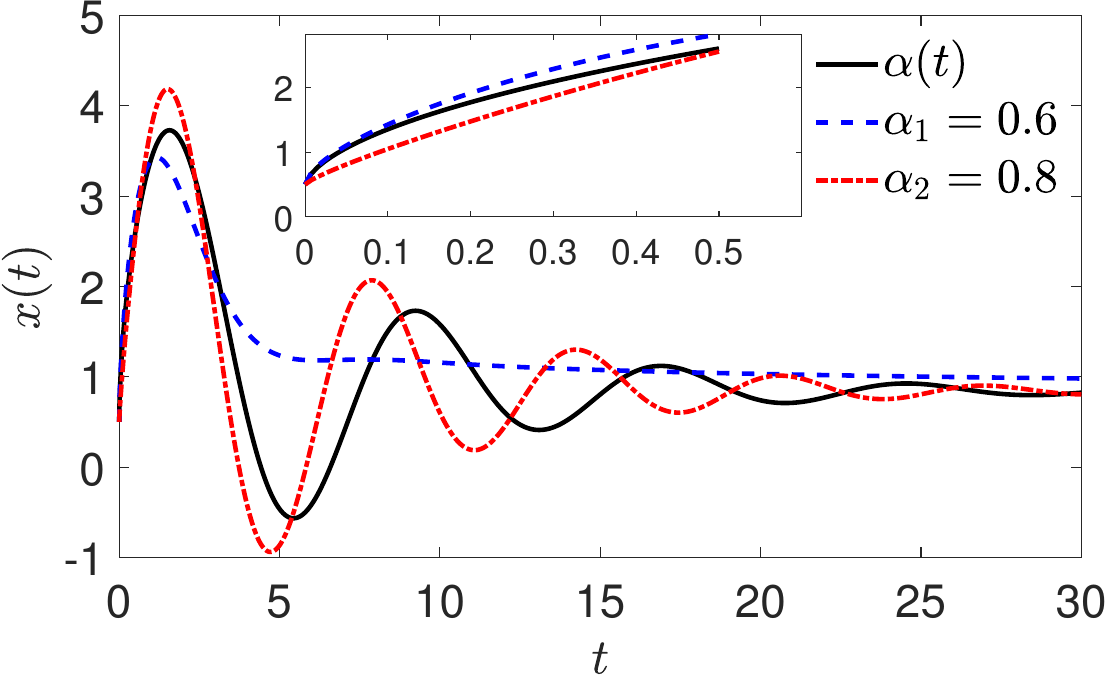} &
			\includegraphics[width=0.45\textwidth]{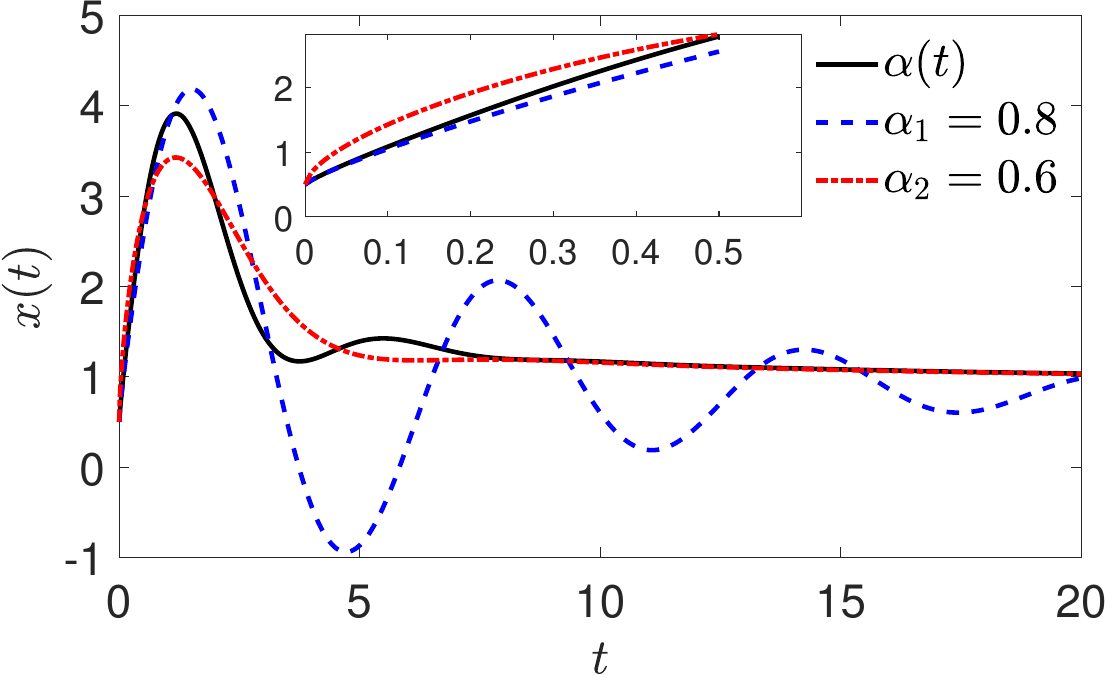} \\
			\includegraphics[width=0.45\textwidth]{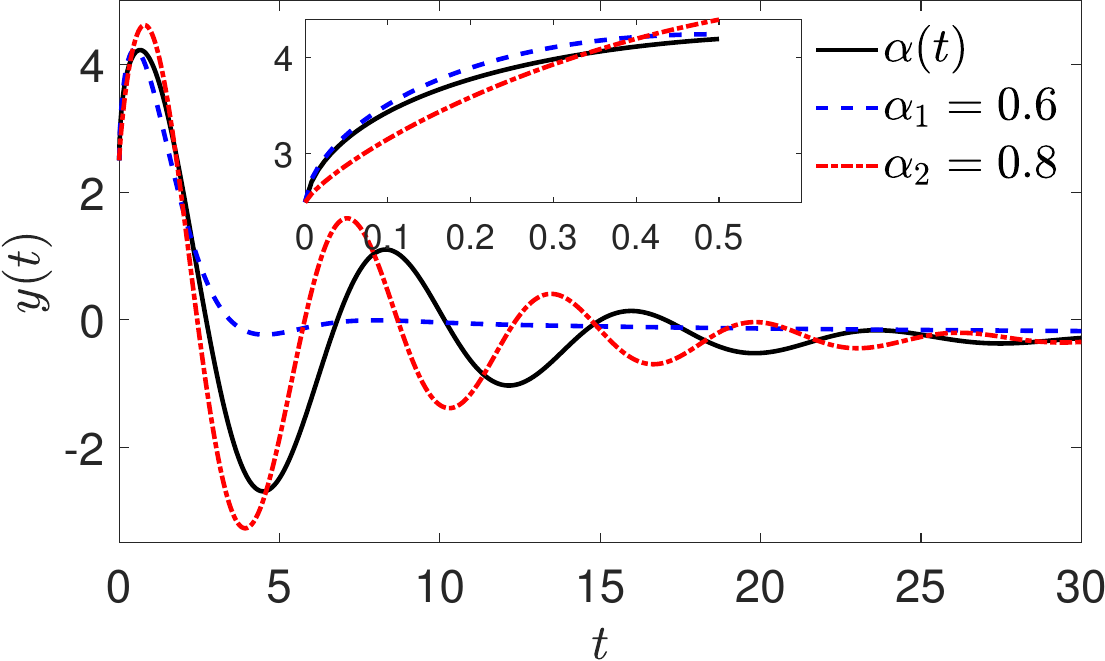} &
			\includegraphics[width=0.45\textwidth]{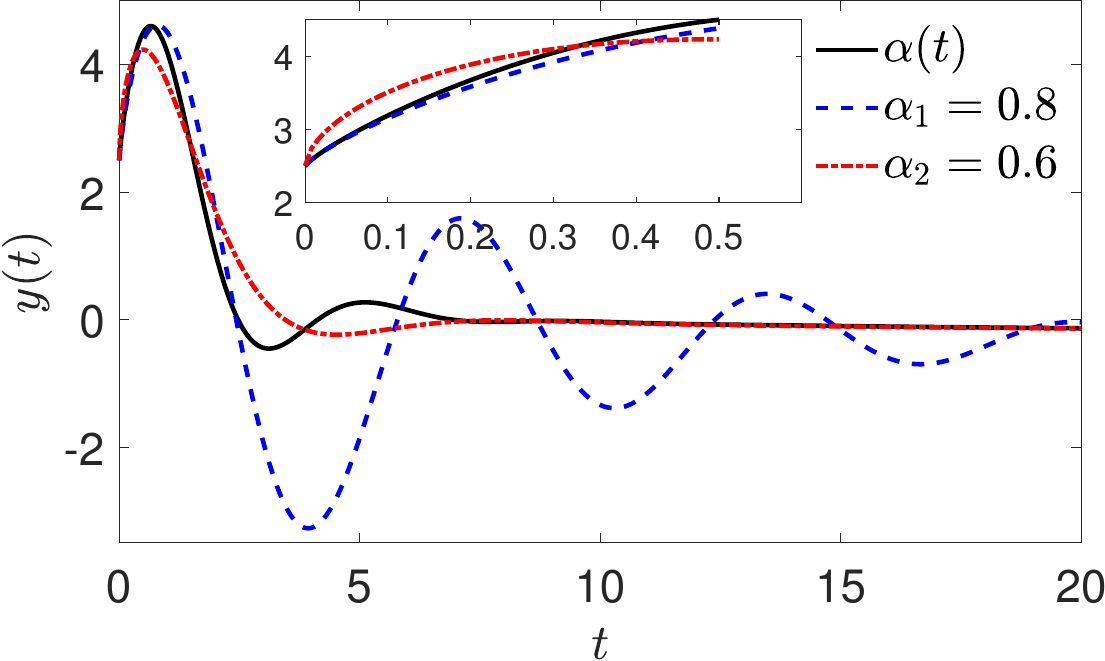} \\
		\end{tabular}
		\caption{Solutions of the linear system \eqref{eq:LinearSystem} for different values of $\alpha_1$ and $\alpha_2$ with $c=2.0$.}\label{fig:Fig_LinearSystem}
	}
\end{figure}

{Also with the linear system  \eqref{eq:LinearSystem} the comparison between reference and numerical solutions allows to confirm the first-order convergence of the numerical method under investigation, as shown in Table \ref{tab:LinearSystem_Error_EOC_System}.
}

\begin{table}[ht]
	{\footnotesize
		\[
		\begin{array}{|r|cc|cc|cc|} \hline
			& \multicolumn{2}{|c|}{\alpha_1=0.6 \;\; \alpha_2=0.8}& \multicolumn{2}{|c|}{\alpha_1=0.8 \;\; \alpha_2=0.6}& \multicolumn{2}{|c|}{\alpha_1=0.5 \;\; \alpha_2=0.8}\\ 
			& \multicolumn{2}{|c|}{c=2.0 }& \multicolumn{2}{|c|}{c=2.0 }& \multicolumn{2}{|c|}{c=1.0 }\\ 
			h & \textrm{Error} & \textrm{EOC} & \textrm{Error} & \textrm{EOC} & \textrm{Error} & \textrm{EOC} \\ \hline
			2^{-3} & 2.08(-1) & & 1.80(-2) & & 6.43(-2) & \\ 
			2^{-4} & 1.12(-1) & 0.892 & 8.56(-3) & 1.069 & 3.36(-2) & 0.935 \\ 
			2^{-5} & 5.74(-2) & 0.962 & 3.98(-3) & 1.102 & 1.70(-2) & 0.986 \\ 
			2^{-6} & 2.84(-2) & 1.016 & 1.85(-3) & 1.106 & 8.31(-3) & 1.029 \\ 
			2^{-7} & 1.34(-2) & 1.084 & 8.42(-4) & 1.135 & 3.90(-3) & 1.091 \\ 
			2^{-8} & 5.77(-3) & 1.215 & 3.56(-4) & 1.242 & 1.68(-3) & 1.218 \\ 
			\hline
		\end{array}
		\]
		\normalsize
		\caption{Errors (at $T=8$) evaluated in $\|\cdot\|_2$ and EOCs of the numerical method applied to the linear system \eqref{eq:LinearSystem}} 
		\label{tab:LinearSystem_Error_EOC_System} 
	}
\end{table}

\section{Concluding remarks}
\label{sec:conc}
In this work we have discussed in detail both theory and computation of some VO-FDEs with order transition functions of exponential type. Specifically, first we have investigated the relationships between VO operators and CO one at early and late times. We have presented a numerical method, based on Lubich's CQRs, that extends the backward Euler scheme to the VO case. In particular, we have developed an algorithm for the computation of weights based on a detailed error analysis. Numerical experiments have, further, shown the robustness of the proposed approach.

In order to obtain more accurate numerical solutions, the natural continuation of the proposed research will entail the development of higher-order methods, for which a preliminary analysis of solution regularity at the origin is necessary to properly select regularization weights \cite{Lubich1988a,Lubich1988b,Lubich2004}. Some seeds for this analysis are in Section \ref{S:Propr_TransFunction_Exp} at least for what concerns relaxation equations. However this matter deserves a deeper discussion which is left to a future study.

It will be, moreover, of particular interest to extend the methodology presented in this study to other classes of transition functions $\alpha(t)$. In particular, to address the problem of VO operators with transition functions $\alpha(t)$ for which the analytic expression of their LT is not known (thus to avoid condition A2 in Section \ref{S:VOGeneral}). Clearly, such a problem can only be tackled at the numerical. It is worth mentioning that this research line is mostly uncharted territory to date \cite{WeidemanFornberg2023}, although it has the potential of becoming a stimulating and prolific research topic in the near future.

\section*{Acknowledgments}
The work of R.G. is partially supported by the MIUR under the PRIN2017 project n. 2017E844SL and by the INdAM under the  GNCS Project E53C22001930001. The work of A.G. has been carried out in the framework of activities of the National Group for Mathematical Physics (GNFM, INdAM).
\newpage

\begin{thebibliography}{10}

\bibitem{CuestaKiraneAlsaediAhmad2021}
{\sc E.~Cuesta, M.~Kirane, A.~Alsaedi, and B.~Ahmad}, {\em On the sub-diffusion
  fractional initial value problem with time variable order}, Adv. Nonlinear
  Anal., 10 (2021), pp.~1301--1315.

\bibitem{DarveDeliaGarrappaGiustiRUbio2022}
{\sc E.~Darve, M.~D'Elia, R.~Garrappa, A.~Giusti, and N.~L. Rubio}, {\em On the
  fractional {L}aplacian of variable order}, Fract. Calc. Appl. Anal., 25
  (2022), pp.~15--28.

\bibitem{Diethelm2010}
{\sc K.~Diethelm}, {\em {The Analysis of Fractional Differential Equations}},
  vol.~2004 of Lecture Notes in Mathematics, Springer-Verlag, Berlin, 2010.

\bibitem{DiethlemGarrappaGiustiStynes2020}
{\sc K.~Diethlem, R.~Garrappa, A.~Giusti, and M.~Stynes}, {\em Why fractional
  derivatives with nonsingular kernels should not be used}, Fract. Calc. Appl.
  Anal., 23 (2020), pp.~610--634.

\bibitem{DuSunWang2022}
{\sc R.-L. Du, Z.-Z. Sun, and H.~Wang}, {\em Temporal second-order finite
  difference schemes for variable-order time-fractional wave equations}, SIAM
  J. Numer. Anal., 60 (2022), pp.~104--132.


\bibitem{GarrappaGiustiMainardi2021}
{\sc R.~Garrappa, A.~Giusti, and F.~Mainardi}, {\em Variable-order fractional
  calculus: a change of perspective}, Commun. Nonlinear Sci. Numer. Simul., 102
  (2021), pp.~Paper No. 105904, 16.

\bibitem{Giusti:2020rul}
{\sc A.~Giusti}, {\em {MOND-like Fractional Laplacian Theory}}, Phys. Rev. D,
  101 (2020), 124029.

\bibitem{Giusti:2020kcv}
{\sc A.~Giusti, R.~Garrappa, and G.~Vachon}, {\em {On the Kuzmin model in
  fractional Newtonian gravity}}, Eur. Phys. J. Plus, 135 (2020), 798.

\bibitem{GorenfloAbdelRehim2007}
{\sc R.~Gorenflo and E.~Abdel-Rehim}, {\em Convergence of the
  {G}r\"{u}nwald-{L}etnikov scheme for time-fractional diffusion}, J. Comput.
  Appl. Math., 205 (2007), pp.~871 -- 881.

\bibitem{GorenfloKilbasMainardiRogosin2020}
{\sc R.~Gorenflo, A.~A. Kilbas, F.~Mainardi, and S.~V. Rogosin}, {\em
  Mittag-{L}effler {Functions, Related Topics and Applications}}, Springer
  Monographs in Mathematics, Springer-Verlag Berlin Heidelberg, 2020.

\bibitem{Grove1991}
{\sc A.~C. Grove}, {\em An introduction to the {L}aplace transform and the
  {$z$} transform}, Prentice Hall, Inc., Englewood Cliffs, NJ, 1991.

\bibitem{Hanyga2020}
{\sc A.~Hanyga}, {\em A comment on a controversial issue: a generalized
  fractional derivative cannot have a regular kernel}, Fract. Calc. Appl.
  Anal., 23 (2020), pp.~211--223.

\bibitem{LePage1980}
{\sc W.~R. LePage}, {\em Complex variables and the {L}aplace transform for
  engineers}, Dover Publications, Inc., New York, 1980.
\newblock Corrected reprint of the 1961 original.

\bibitem{Lubich1988a}
{\sc C.~Lubich}, {\em Convolution quadrature and discretized operational
  calculus. {I}}, Numer. Math., 52 (1988), pp.~129--145.

\bibitem{Lubich1988b}
{\sc C.~Lubich}, {\em Convolution quadrature and discretized operational
  calculus. {II}}, Numer. Math., 52 (1988), pp.~413--425.

\bibitem{Lubich2004}
{\sc C.~Lubich}, {\em Convolution quadrature revisited}, BIT, 44 (2004),
  pp.~503--514.

\bibitem{Luchko2020_FCAA}
{\sc Y.~Luchko}, {\em Fractional derivatives and the fundamental theorem of
  fractional calculus}, Fract. Calc. Appl. Anal., 23 (2020), pp.~939--966.

\bibitem{Luchko2021_Mathematics}
{\sc Y.~Luchko}, {\em General fractional integrals and derivatives with the
  {S}onine kernels}, Mathematics, 9 (2021), p.~594.

\bibitem{Luchko2021_FCAA}
{\sc Y.~Luchko}, {\em Operational calculus for the general fractional
  derivative and its applications}, Fract. Calc. Appl. Anal., 24 (2021),
  pp.~338--375.

\bibitem{LuchkoYamamoto2020}
{\sc Y.~Luchko and M.~Yamamoto}, {\em The general fractional derivative and
  related fractional differential equations}, Mathematics, 8 (2020), p.~2115.

\bibitem{OngunArslanGarrappa2013}
{\sc M.~Y.~t. Ongun, D.~Arslan, and R.~Garrappa}, {\em Nonstandard finite
  difference schemes for a fractional-order {B}russelator system}, Adv.
  Difference Equ.,  (2013), pp.~2013:102, 13.

\bibitem{Podlubny1999}
{\sc I.~Podlubny}, {\em Fractional differential equations}, Academic Press
  Inc., San Diego, CA, 1999.

\bibitem{Rasof1962}
{\sc B.~Rasof}, {\em The initial- and final-value theorems in {L}aplace
  transform theory}, J. Franklin Inst., 274 (1962), pp.~165--177.

\bibitem{Samko1995a}
{\sc S.~G. Samko}, {\em Fractional integration and differentiation of variable
  order}, Anal. Math., 21 (1995), pp.~213--236.

\bibitem{SamkoCardoso2003b}
{\sc S.~G. Samko and R.~P. Cardoso}, {\em Integral equations of the first kind
  of {S}onine type}, Int. J. Math. Math. Sci.,  (2003), pp.~3609--3632.

\bibitem{SamkoCardoso2003a}
{\sc S.~G. Samko and R.~P. Cardoso}, {\em Sonine integral equations of the
  first kind in {$L_p(0,b)$}}, Fract. Calc. Appl. Anal., 6 (2003),
  pp.~235--258.

\bibitem{SamkoKilbasMarichev1993}
{\sc S.~G. Samko, A.~A. Kilbas, and O.~I. Marichev}, {\em Fractional integrals
  and derivatives}, Gordon and Breach Science Publishers, Yverdon, 1993.

\bibitem{SamkoRoss1993}
{\sc S.~G. Samko and B.~Ross}, {\em Integration and differentiation to a
  variable fractional order}, Integral Transform. Spec. Funct., 1 (1993),
  pp.~277--300.

\bibitem{Scarpi1972a}
{\sc G.~Scarpi}, {\em Sopra il moto laminare di liquidi a viscosist\`a
  variabile nel tempo}, Atti. Accademia delle Scienze, Isitituto di Bologna,
  Rendiconti (Ser. XII), 9 (1972), pp.~54--68.

\bibitem{Scarpi1972b}
{\sc G.~Scarpi}, {\em Sulla possibilit\`a di un modello reologico intermedio di
  tipo evolutivo}, Atti Accad. Naz. Lincei Rend. Cl. Sci. Fis. Mat. Nat. (8),
  52 (1972), pp.~912--917 (1973).

\bibitem{SchererKallaTangHuang2011}
{\sc R.~Scherer, S.~L. Kalla, Y.~Tang, and J.~Huang}, {\em The
  {G}r\"{u}nwald-{L}etnikov method for fractional differential equations},
  Comput. Math. Appl., 62 (2011), pp.~902--917.

\bibitem{Sonine1884}
{\sc N.~Sonine}, {\em Sur la g\'{e}n\'{e}ralisation d'une formule d'{A}bel},
  Acta Math., 4 (1884), pp.~171--176.

\bibitem{SunChenLiChen2022}
{\sc H.~Sun, W.~Chen, C.~Li, and Y.~Chen}, {\em Finite difference schemes for
  variable-order time fractional diffusion equation}, Internat. J. Bifur. Chaos
  Appl. Sci. Engrg., 22 (2012), pp.~1250085, 16.

\bibitem{TrefethenWeideman2014}
{\sc L.~N. Trefethen and J.~A.~C. Weideman}, {\em The exponentially convergent
  trapezoidal rule}, SIAM Rev., 56 (2014), pp.~385--458.

\bibitem{WangLi2007}
{\sc Y.~Wang and C.~Li}, {\em Does the fractional {B}russelator with efficient
  dimension less than 1 have a limit cycle?}, Phys. Lett. A: Gen. At. Solid
  State Phys., 363 (2007), pp.~414--419.

\bibitem{WeidemanFornberg2023}
{\sc J.~A.~C. Weideman and B.~Fornberg}, {\em Fully numerical {L}aplace
  transform methods}, Numer. Algorithms, 92 (2023), pp.~985--1006.

\bibitem{WeidemanTrefethen2007}
{\sc J.~A.~C. Weideman and L.~N. Trefethen}, {\em Parabolic and hyperbolic
  contours for computing the {B}romwich integral}, Math. Comp., 76 (2007),
  pp.~1341--1356.

\bibitem{ZayernouriKarniadakis2015}
{\sc M.~Zayernouri and G.~E. Karniadakis}, {\em Fractional spectral collocation
  methods for linear and nonlinear variable order {FPDE}s}, J. Comput. Phys.,
  293 (2015), pp.~312--338.

\bibitem{ZengZhangKarniadakis2015}
{\sc F.~Zeng, Z.~Zhang, and G.~E. Karniadakis}, {\em A generalized spectral
  collocation method with tunable accuracy for variable-order fractional
  differential equations}, SIAM J. Sci. Comput., 37 (2015), pp.~A2710--A2732.

\bibitem{ZhaoSunKarniadakis2015}
{\sc X.~Zhao, Z.-Z. Sun, and G.~E. Karniadakis}, {\em Second-order
  approximations for variable order fractional derivatives: algorithms and
  applications}, J. Comput. Phys., 293 (2015), pp.~184--200.

\bibitem{ZhengWang2020}
{\sc X.~Zheng and H.~Wang}, {\em An error estimate of a numerical approximation
  to a hidden-memory variable-order space-time fractional diffusion equation},
  SIAM J. Numer. Anal., 58 (2020), pp.~2492--2514.
 
\bibitem{Zheng2022}
{\sc X.~Zheng}, {\em Approximate inversion for {A}bel integral operators of
	variable exponent and applications to fractional {C}auchy
	problems},
Fract. Calc. Appl. Anal., 25 (20220), pp.~1585--1603.
  
 

\bibitem{ZhuangLiuAnhTurner2009}
{\sc P.~Zhuang, F.~Liu, V.~Anh, and I.~Turner}, {\em Numerical methods for the
  variable-order fractional advection-diffusion equation with a nonlinear
  source term}, SIAM J. Numer. Anal., 47 (2009), pp.~1760--1781.

\end{thebibliography}
\end{document}